\documentclass[12pt]{amsart}
\usepackage{latexsym,amsmath, amscd, amssymb, amsthm,  euscript, mathrsfs, hyperref}
\usepackage{amssymb, paralist, xspace, graphicx, url, amscd, euscript, mathrsfs, stmaryrd}
\usepackage[dvipsnames]{xcolor}
\usepackage[all]{xy}
\usepackage{tikz}
\usetikzlibrary{arrows}
\usetikzlibrary{cd}
\usepackage{makecell}

\usepackage[english]{babel}

\usepackage[letterpaper,top=2cm,bottom=2cm,left=3cm,right=3cm,marginparwidth=1.75cm]{geometry}

\usepackage{amsmath}
\usepackage{graphicx}
\usepackage[all]{xy}
\usepackage[utf8]{inputenc}
\usepackage{dsfont}
\usepackage{amsfonts}
\usepackage{mathrsfs}
\usepackage{mathtools}

\newtheorem{thm}[subsection]{Theorem}

\newtheorem{lem}[subsection]{Lemma}
\newtheorem{prop}[subsection]{Proposition}
\newtheorem{cor}[subsection]{Corollary}
\theoremstyle{definition}

\newtheorem{ex}[subsection]{Example}
\newtheorem{defn}[subsection]{Definition}

\newtheorem{rmk}[subsection]{Remark}
\newtheorem{defprop}[subsection]{Definition/ Proposition}

\let\saveqed\qed
\renewcommand\qed{%
   \ifmmode\displaymath@qed
   \else\saveqed
   \fi}
\newcommand{\Pic}{\mathrm{Pic}}
\newcommand{\Spec}{\mathrm{Spec} \ }
\newcommand{\uSpec}{\underline{\mathrm{Spec}}}
\newcommand{\Proj}{\mathrm{Proj} \ }
\newcommand{\Qcoh}{\mathrm{Qcoh}}

\newcommand{\Isom}{\mathrm{Isom}}
\newcommand{\Aut}{\mathrm{Aut}}
\newcommand{\uIsom}{\underline{\mathrm{Isom}}}
\newcommand{\uAut}{\underline{\mathrm{Aut}}}

\newcommand{\sGp}{\mathrm{ShGroup}}
\newcommand{\Ba}{\mathrm{Band}}

\newcommand{\Hom}{\mathrm{Hom}}
\newcommand{\Hex}{\mathrm{Hex}}
\newcommand{\Int}{\mathrm{Int}}
\newcommand{\pr}{\mathrm{pr}}

\newcommand{\ms}{\mathscr}
\newcommand{\cha}{\mathrm{char}}
\newcommand{\thickslash}{\mathbin{\!\!\pmb{\fatslash}}}

\newcommand{\Z}{\mathbb{Z}}

\counterwithin{equation}{subsection}

\title{Picard Groups of Stacky Curves}
\author{Rose Lopez}

\begin{document}
\maketitle

\begin{abstract}
We describe the Picard group of a tame stacky curve as an extension of two groups, which depend on the gerbe class of the curve over its rigidification, a stacky curve with trivial generic stabilizer, and the residual gerbes of the rigidification. After presenting the general formulation, we compute some specific examples. In particular, we provide an example highlighting the fact that the extension of the groups truly depends on the class of the gerbe. 
\end{abstract}

\section{Introduction}

In his famous 1965 paper \textit{Picard Groups of Moduli Problems} \cite{Mumfordart}, Mumford computed the Picard group of the moduli stack of elliptic curves over fields of characteristic different from 2 and 3.  This work was subsequently revisited by Fulton and Olsson \cite{PicM}, and some generalizations to certain modular curves were studied by Niles \cite{Niles}.  All of these works use the moduli interpretation of the stacks in question in essential ways.  The purpose of this article is to give a general method for computing the Picard group of an arbitrary tame stacky curve. Precisely, a stacky curve is a smooth, proper, geometrically connected Deligne-Mumford stack of dimension $1$ over a field.

There are two key cases we need to consider. The first case is for a tame stacky curve $\ms{Y}$ with trivial generic stabilizer. Here we compute the Picard group exactly in terms of the Picard group of its coarse moduli space and the stabilizer groups of the stacky points. The second case is for a gerbe $\ms X$ over a stacky curve with trivial generic stabilizer $\ms{Y}$, and here we are only able to compute the Picard group as an extension of two groups, one group being $\Pic (\ms{Y})$. We prove the following two theorems for these two cases. 

\begin{thm}
    Let $\ms{Y}$ be a tame stacky curve with trivial generic stabilizer and coarse moduli space $\pi:\ms{Y}\to X$. Let $x_1,\ldots, x_r$ be the stacky points of $\ms{Y}$ with stabilizer groups $\mu_{n_i}$, cyclic of order $n_i$. Then we have an exact sequence 
    $$\xymatrix{0 \ar[r] & \Pic(X)  \ar[r]^-{\pi^*} & \Pic(\ms{Y}) \ar[r]^-{\chi} & \prod_{i=1}^r \Z/n_i\Z \ar[r] & 0},$$
    where $\chi= (\chi_{x_1^*(-)},\ldots, \chi_{x_r^*(-)})$ describes the representations of the inertia on a line bundle pulled back to the residual gerbes at the stacky points.  
    Moreover, 
    $\Pic(\ms{Y})$ is the pushout of the diagram 
    $$\xymatrix{\Z^r \ar[r]^-{(\cdot n_1, \ldots, \cdot n_r)} \ar[d]^-{\phi} & \Z^r \ar[d] \\ \Pic (X) \ar[r]^-{\pi^*} & \Pic (\ms{Y})}$$
    where $\phi: \Z^r\to \Pic (X)$ is given by $e_i\mapsto I_{\pi(x_i)}$, for $e_i=(0, \ldots, 1, \ldots, 0)$ the $i$th basis vector, and $I_{\pi(x_i)}$ the ideal sheaf of $\ms{O}_{X}$-modules of the point $\pi(x_i)\in X$. 
\end{thm}

\begin{thm}
    Let $\ms{Y}$ be a stacky curve with trivial generic stabilizer and let $\ms{H}$ be a tame locally constant finitely presented band on $\ms{Y}$. Let $p:\ms{X}\to \ms{Y}$ be a gerbe banded by $\ms{H}$, with cohomology class $[\ms{X}]\in H^2(\ms{Y},\ms{H})$. We have an exact sequence
    $$\xymatrix{ 0\ar[r] & \Pic(\ms{Y}) \ar[r]^-{p^*}&  \Pic (\ms{X}) \ar[r]^-{\chi} & \mathrm{Hom}(\ms{H}, \mathbb{G}_m) \ar[r]^-{(-)_*[\ms{X}]} & H^2(\ms{Y},\mathbb{G}_m)},$$
    where $\chi$ describes the representation of $\ms{H}$ acting on the line bundle, and $(-)_*[\ms{X}]$ sends $(\epsilon:\ms{H}\to\mathbb{G}_m)$ to the image of $[\ms X]$ under the cohomology pushforward map $$\epsilon_*:H^2(\ms{Y},\ms{H})\to H^2(\ms{Y},\mathbb{G}_m).$$
\end{thm}
\begin{rmk}
    A band is locally a sheaf of groups, with gluing data only well-defined up to conjugation. By a tame locally constant finitely presented band, we mean that locally, the band comes from a tame finite constant sheaf of groups. When $\ms H$ can not be identified with a sheaf of groups, $\epsilon_*$ is more complicated to define. We describe the stack of bands and the construction of $\epsilon_*$ in Section 4. 
\end{rmk}
\begin{rmk}
    Theorem 1.2 holds more generally for $\ms Y$ a locally finitely presented algebraic stack over a locally noetherian base, as its proof will show. 
\end{rmk}

Given an arbitrary tame stacky curve $\ms{X}$, we can combine case one and case two to give a general description of $\Pic (\ms{X})$. The key point needed is to realize $\ms{X}$ as a gerbe over some stacky curve with trivial generic stabilizer $\ms{Y}$ using rigidification, in which a subgroup $\ms{H}$ of the inertia $\ms{I}(\ms{X})$ is removed from $\ms{X}$. Rigidification allows $\ms{X}$ to become a gerbe banded by $\underline{\ms{H}}$, the band associated to $\ms{H}$, over its rigidification with respect to $\ms{H}$, $\ms{X}\thickslash \ms{H}$. (See Section 4, and \cite[Appendix A]{Tame}.) We have the following proposition.

\begin{prop}
    Let $\ms{X}$ be a stacky curve with inertia group $\ms{I}$. Then there exists a finite flat subgroup stack $\ms{H}\leq \ms{I}$ such that the rigidification of $\ms{X}$ with respect to $\ms{H}$ is a stacky curve with trivial generic stabilizer $\ms{Y}:= \ms{X}\thickslash \ms{H}$, and the natural map $p:\ms{X}\to \ms{Y}$ is a gerbe banded by $\underline{\ms{H}}$. Moreover, if $\ms X$ is tame, then so is $\ms Y$. 
\end{prop}

\begin{rmk}
    The proof of this proposition mainly requires finding the finite flat subgroup stack $\ms{H}$. The result that given a finite flat subgroup $\ms{H}$ of $\ms{I}$, there exists a locally finitely presented algebraic stack $\ms{X}\thickslash \ms{H}$ and map $\ms{X}\to \ms{X}\thickslash \ms{H}$ such that $\ms{X}$ is a gerbe banded by $\underline{\ms{H}}$ over $\ms{X}\thickslash \ms{H}$ appears in \cite{Tame}. 
\end{rmk}

Putting all of the pieces together, we have the following corollary.

\begin{cor}
    Let $\ms{X}$ be a tame stacky curve with coarse moduli space $\pi:\ms{X}\to X$ and inertia stack $\ms{I}$. Let $\ms{H}\leq \ms{I}$ be as in Proposition 1.5, so that $\ms{Y}:=\ms{X}\thickslash\ms{H}$ is a tame stacky curve with trivial generic stabilizer and $p:\ms{X}\to\ms{Y}$ is a gerbe banded by $\underline{\ms{H}}$. Let $[\ms{X}]\in H^2(\ms{Y},\underline{\ms{H}})$ be the gerbe class of $\ms{X}$. We have an exact sequence
    $$\xymatrix{ 0\ar[r] & \Pic(\ms{Y}) \ar[r]^-{p^*}&  \Pic (\ms{X}) \ar[r]^-{\chi} & \mathrm{Hom}(\underline{\ms{H}}, \mathbb{G}_m) \ar[r]^-{(-)_*[\ms{X}]} & H^2(\ms{Y},\mathbb{G}_m)},$$
    where $\chi$ and and $(-)_*[\ms{X}]$ are as in Theorem 1.2, and $\Pic (\ms{Y})$ is described as in Theorem 1.1.
\end{cor}

\subsection{Organization} The paper is organized as follows. In Section 2, we will begin with some conventions and relevant results on stacks. We then introduce stacky curves, and explain some relevant results about stacky curves. In particular, we will see that a tame stacky curve with trivial generic stabilizer has an explicit quotient structure locally, allowing us to realize any stacky curve with trivial generic stabilizer as a root stack.  

In Section 3, we prove Theorem 1.1, using the root stack structure and certain line bundles on root stacks. 

In Sections 4 and 5, we give some background on gerbes, the stack of bands relevant for nonabelian gerbes, and rigidification. 

In Section 6, we prove Theorem 1.2. Rigidification allows us to first reduce to the case of an abelian gerbe. 

In Section 7, we prove Proposition 1.5 and Corollary 1.7. 

In Section 8, we consider several examples. We begin with the moduli stack of elliptic curves $\ms{M}_{1,1}$, as considered in \cite{PicM}. In general, if the gerbe class of $\ms{X}$ is trivial, then for any character $\epsilon:A\to\mathbb{G}_m$, $\epsilon_*[\ms{X}]$ will also be trivial, and $\chi$ is surjective in this case. We will see an example where the gerbe class of $\ms{X}$ is not trivial, but $\chi$ is still surjective, and we will also see an example where $\chi$ is not surjective. In particular, we will find two curves $\ms{X}$ and $\ms{X}'$ such that $\ms{Y}\cong \ms{Y}'$ and $\ms{H}\cong \ms{H}'$, but $\Pic\ms X\not\cong \Pic \ms X'$. The question of whether $\chi$ is surjective is related to existence of twisted line bundles on $\ms{X}$, which is in general a difficult question and depends on the gerbe class of $\ms{X}$.

\subsection{Acknowledgements}
The author thanks Martin Olsson for this project idea and many helpful discussions and edits of the paper. This work was partially supported by NSF RTG grant DMS-1646385.

\section{Conventions on Stacks}

In this section, we give some relevant definitions and results on stacks, stacky curves, and root stacks. Everything in this section is well-known, but included for the reader's convenience and to set-up notation.
\begin{defn}
    Let $S$ be a scheme. An \textit{algebraic stack $\ms{X}$ over $S$}, as in \cite[026O]{SP}, is a category fibered in groupoids $p:\ms{X}\to (\mathrm{Sch}/S)_{fppf}$ whose diagonal morphism is representable by algebraic spaces, and such that there exists a smooth surjective covering $X\to \ms X$ by a scheme $X$.
\end{defn}
An algebraic stack $\ms{X}/S$ has the property that for all $S$-schemes $T$ and all $x,y\in\ms{X}(T)$, the presheaf
$$\uIsom_T(x,y): (\mathrm{Sch}/S)_{fppf}|_T^{\mathrm{op}}\to \mathrm{Set},$$
$$(f:T'\to T)\mapsto \Isom_{\ms{X}(T')}(f^*x,f^*y)$$
is an algebraic space. In particular, $\uAut_T(x):=\uIsom_T(x,x)$ is an algebraic space as well.

In this paper, by a stack, we mean an algebraic stack. Unless otherwise stated, the stacks studied in this paper will have finite diagonal and will be of finite presentation over $S$, which implies that $\uIsom_T(x,y)$ and $\uAut_T(x)$ are finite $T$-schemes of finite presentation \cite[7.2.10]{Olssonbook}. 

\begin{defn}
    For any geometric point $x:\Spec k\to \ms{X}$ of $\ms{X}$, the sheaf of groups $G_x:=\uAut(x):=\uAut_k(x)$ is called the \textit{stabilizer group} of $x$, and if $G_x$ is nontrivial, then we call $x$ a \textit{stacky point}. In this case, $x$ factors through the inclusion $BG_x \hookrightarrow \ms{X}$ of $BG_x=[\Spec k/G_x]$ as a strictly full substack of $\ms{X}$, and we call $BG_x$ the \textit{residual gerbe} at $x$ \cite[06ML, 06UH]{SP}.
\end{defn}

\begin{defn}
    A \textit{Deligne-Mumford stack} is an algebraic stack such that there exists an \'etale surjection $U\to \ms X$ from a scheme $U$. 
\end{defn}

\begin{lem}
    Let $\ms X$ be a finite-type Deligne-Mumford stack over a scheme $S$. Then for all geometric points $x:\Spec k\to \ms X$, $G_x$ is a finite constant group scheme, and the diagonal has finite discrete fibers.
\end{lem}
\begin{proof}
    A Deligne-Mumford stack has formally unramified diagonal \cite[8.3.3]{Olssonbook}. Since $\ms X$ is finite-type, the diagonal is unramified. Thus, all $G_x$ are unramified group schemes over algebraically closed fields, so have diagonal morphism an open immersion \cite[I.3.1]{SGA1}. Then all $G_x$ are reduced, and therefore \'etale, hence finite and discrete. The fibers of the diagonal are $\uIsom_k(x_1,x_2)$ for a geometric point $(x_1,x_2):\Spec k\to \ms X\times_S \ms X$, which is either empty or isomorphic to the finite discrete $\uAut(x_1)=G_{x_1}$.
\end{proof}

\begin{defn}
    If $G_x$ is a finite group scheme, then we say $x$ is a \textit{tame} point if $\cha k\nmid \#G_x(k)$.
\end{defn}

\begin{prop}
    For an algebraic stack $\ms X$ with finite diagonal, there exists an open substack $\ms U \subseteq \ms X$ characterized by the property that for any $T$-point $t:T\to \ms X$, $t$ factors through $\ms U$ if and only if $\uAut_T(t)\cong \{\mathrm{id}_T\}$. 
\end{prop}

\begin{proof}
    Let $\pi: X\to \ms X$ be a smooth cover from a scheme, and let $I\to X$ denote the fiber product $\ms I (\ms X) \times_{\ms X} X$. Since $\ms X$ has finite diagonal, $\ms I (\ms X)\to \ms X$ is finite, and thus $I$ is finite over $X$, so $I=\uSpec_X(\ms A)$ for some coherent sheaf of $\ms O_X$-algebras \cite[10.2.4]{Olssonbook}. Define 
    $$U=\{x\in X: \mathrm{dim}_{\kappa(x)}(\ms A_x\otimes_{\ms O_X} \kappa(x))=1\}.$$
    Then $U$ is an open subscheme in $X$ by semi-continuity \cite[II.Exercise 5.8]{Hartshorne}. 
   
    For any $T$-point $t:T\to \ms X$, let $t':T'\to X$ be the base change to $X$. We have the following commutative diagram with the outer and right squares cartesian,
    $$\xymatrix{\uAut_{T'}(\pi\circ t') \ar[r] \ar[d] & I \ar[r] \ar[d] & \ms I \ar[d] \\ T' \ar[r]^-{t'} & X \ar[r]^-{\pi} & \ms X},$$
    so the left square is cartesian as well. 
    Then $t'$ factors through $U$ if and only if for every point $x$ in the image of $t'$, $\mathrm{dim}_{\kappa(x)}(\ms A_x\otimes_{\ms O_X} \kappa(x))=1$, which holds if and only if $\uSpec_{\kappa(x)}(\ms A_x\otimes_{\ms O_X} \kappa(x))\cong \Spec \kappa(x)$, that is, $I_{\kappa(x)}\cong X_{\kappa(x)}$ for every $x$ in the image of $t'$, if and only if $\uAut_{T'}(\pi\circ t')\cong \{\mathrm{id}_{t'}\}$.
     
    By effectivity of descent for open subschemes, we have an equivalence of categories between open sets in $\ms X$ and pairs $(U,\sigma)$ of an open set $U\subseteq X$ and an isomorphism $\sigma:\pr_1^*U\cong \pr_2^*U$, where $pr_i$ are the projections $X\times_{\ms X} X\to X$, such that $\sigma$ satisfies the cocycle condition in $X\times_{\ms X} X\times_{\ms X} X$. For $U$ defined above, we have that $\pr_1^*U=\pr_2^*U$. Indeed, $\pr_i^*U$ is given by 
    $$\pr_i^*U=\{(x_1,x_2,\sigma): x_1,x_2\in X, x_i\in U, \sigma:\pi(x_1)\cong \pi(x_2)\},$$
    and if $x\in U$ such that $\pi(x)\cong \pi(y)$, then we have an isomorphism $\uAut(\pi(x))\cong \uAut(\pi(y))$ (see Remark 4.2), so $y\in U$ as well. The equality $\pr_1^*U=\pr_2^*U$ clearly satisfies the cocycle condition, so descent gives us an open substack $\ms U\subseteq \ms X$ such that for any $T$-point $t:T\to \ms X$, $t$ factors through $\ms U$ if and only if $\uAut_T(t)\cong \{\mathrm{id}_t\}$.
\end{proof}

\begin{defn}
    We say that $\ms{X}$ has \textit{trivial generic stabilizer} if the open substack $\ms U$ is dense in $\ms X$.
\end{defn}

\begin{defprop}
    For any algebraic stack $\ms{X}$ locally of finite presentation with finite diagonal over a locally noetherian base $S$, there exists a \textit{coarse moduli space} $X$ and a morphism $\pi:\ms{X}\to X$, \cite[1.1]{KM}, \cite[11.1.2]{Olssonbook}, where $X$ is an algebraic space over $S$ and $\pi$ satisfies
    \begin{enumerate}
        \item For any morphism $g:\ms{X}\to Z$ from $\ms{X}$ to an algebraic space $Z$, there exists a unique morphism $f:X\to Z$ such that $g=f\circ \pi$,
        \item For any algebraically closed field $k$, the map $|\ms{X}(k)|\to X(k)$ is bijective, where $|\ms X(k)|$ denotes the set of isomorphism classes in $\ms X(k)$. 
    \end{enumerate} 
    Additionally, $\pi$ is proper, the map $\ms{O}_X\to \pi_*\ms{O}_{\ms{X}}$ is an isomorphism, and for any flat morphism of algebraic spaces $X'\to X$, $\pi':\ms{X}\times_X X'\to X'$ is a coarse moduli space for $\ms{X}'$. 
\end{defprop}

For Deligne-Mumford stacks with a coarse moduli space, we have a very nice local structure theorem.

\begin{thm}\cite[11.3.1]{Olssonbook} 
    Let $\ms{X}$ be a Deligne-Mumford stack over a locally noetherian scheme $S$, locally of finite type and with finite diagonal. Let $\pi:\ms{X}\to X$ be its coarse moduli space. Let $x:\Spec k \to \ms{X}$ be a geometric point with stabilizer group $G_x$. Then there exists an \'etale neighborhood $U\to X$ of the image of $x$ and a finite $U$-scheme $V\to U$ with an action of $G_x$ such that $\ms{X}\times_X U\cong [V/G_x]$. 
\end{thm}

\begin{rmk}
    This theorem will be quite useful to us, as any property of such a Deligne-Mumford stack that can be checked locally can be checked on a stack quotient of a scheme by a finite group. Moreover, we can take the scheme to be affine. Indeed, for any affine open $\Spec A\subset U$, $\Spec A \times_X \ms{X}=[\Spec B/G_x]$, where $A\to B$ is a finite ring map. Also, $[\Spec B/G_x]\to\Spec A$ is a coarse moduli space, so $A=B^{G_x}$. We will often take $A=\widehat{\ms{O}}_{U,x}$, the complete local ring of $U$ at $x$, in which case $B=\widehat{\ms{O}}_{V,x}$. 
\end{rmk}

\subsection{Stacky Curves} 
\begin{defn} 
    A \textit{stacky curve $\ms{X}$ over $k$} is a smooth, proper, geometrically connected Deligne-Mumford stack of dimension 1 over a field. 
\end{defn}

\begin{lem}
    Let $\ms{X}$ be a stacky curve. Then the diagonal is finite, and $\ms X$ has a coarse moduli space $\pi:\ms X\to X$ with $X$ a scheme.
\end{lem}
\begin{proof}
    Since $\ms{X}$ is proper, the diagonal is proper and thus of finite type, so we have that the diagonal is locally quasi-finite if and only if its fibers are discrete. \cite[06RW, 04XB]{SP} Lemma 2.4 shows that fibers are discrete, so the diagonal is locally quasi-finite. Being finite type implies being quasi-compact, so the diagonal is quasi-finite. Since the diagonal is also proper, it is finite \cite[02LS, 0418, 03HA]{SP}.

    Since $\ms X$ is locally of finite type with finite diagonal, $\ms{X}$ has a coarse moduli space $\pi:\ms{X}\to X$. The coarse space is smooth, separated \cite[1.3]{KM} and 1-dimensional over a field, so $X$ is a scheme \cite[V.4.9]{Knutsonbook}.
\end{proof}

\begin{rmk}
Since the diagonal morphism is finite, the diagonal is in fact representable by affine schemes, since for any affine scheme $\Spec A$ mapping into $\ms{X}\times_S \ms{X}$, the fiber product with the diagonal morphism will be an affine scheme $\Spec B$, with $B$ a finite $A$-algebra. 
\end{rmk}

\begin{prop}
    Let $\ms X$ be a stacky curve with trivial generic stabilizer. Then the dense open $\ms U\subseteq \ms X$ in Proposition 2.6 is a dense open subscheme in $\ms X$.
\end{prop}
\begin{proof}
    The dense open $\ms U$ is a smooth, separated, 1-dimensional algebraic space over a field, so is a scheme \cite[V.4.9]{Knutsonbook}.
\end{proof}
\begin{defn}
    A stacky curve $\ms{X}$ is \textit{tame} if every geometric point is a tame point.
\end{defn}

\begin{prop} \cite[2.2 (i)]{Log} Let $\ms{X}$ be a tame stacky curve with trivial generic stabilizer and $\pi:\ms{X}\to X$ its coarse space, and let $x:\Spec k\to X$ be a geometric point. Then $$\Spec \ms O_{X,x}\times_X \ms X\cong [\Spec(\ms O_{X,x}[z]/(z^n-f))/\mu_n]$$ 
where $\ms O_{X,x}$ is the \'etale local ring of $x$, $f\in\ms O_{X,x}$ is a local coordinate at $x$, generating the maximal ideal of $\ms O_{X,x}$, and $\mu_n$ acts trivially on $\ms O_{X,x}$ and acts by multiplication by a primitive $n$th root of unity on $z$. In particular, the stabilizer group of $x$ is a finite cyclic group $\mu_n$ for some $n\in\Z$.
\end{prop} 

\subsection{Root Stacks}
In this subsection, we describe the root stack construction, following \cite[10.3]{Olssonbook}. Essentially, the $n$th root stack of a scheme $X$ along a divisor $D$ on $X$ is the universal object $\ms{X}_n$ mapping to $X$ such that there is an $n$th root of $D$ pulled back to $\ms{X}_n$ on $\ms{X}_n$. In general, for a scheme $X$ and an effective Cartier divisor $D$ on $X$, it may not be possible to take an $n$th root of $D$ on $X$. Whenever this is not possible, stacky structure appears in the root stack. We have the following proposition relating stacky curves to root stacks.

\begin{prop} 
    Let $\ms X$ be a tame stacky curve with trivial generic stabilizer, course space $\pi:\ms X \to X$, and stacky points $x_1, \ldots, x_r$ with stabilizer groups $\mu_{n_i}$. Let 
    $$\ms X_{\mathbf{n}}=\ms X_{n_1}\times_X \cdots \times_X \ms X_{n_r}$$
    where $\ms X_{n_i}$ is the $n_i$th root stack of $X$ at $\pi(x_i)$. Then $\ms X\cong \ms X_{\mathbf{n}}$.
\end{prop}
\begin{proof}
    Follows from \cite[1]{GS}.
\end{proof}

\begin{defn}
    Let $X$ be a scheme. A \textit{generalized effective Cartier divisor} on $X$ is a pair $(L,\rho)$, where $L$ is a line bundle on $X$ and $\rho: L \to \ms O_X$ is a morphism of $\ms O_X$-modules. An isomorphism $\sigma:(L',\rho')\cong (L,\rho)$ of generalized effective Cartier divisors on $X$ is an isomorphism of line bundles $\sigma: L'\cong L$ such that the diagram 
    $$\xymatrix{L' \ar[rr]^-{\sigma} \ar[rd]_-{\rho'}&& L \ar[ld]^-{\rho} \\ &\ms O_x }$$
    commutes. A product structure on generalized effective Cartier divisors is defined by 
$$(L,\rho)\cdot (L',\rho')=(L\otimes L', \rho\otimes \rho'),$$
where $\rho\otimes \rho': L\otimes L'\to \ms O_X\otimes_{\ms O_X} \ms O_X\cong \ms O_X$. For a morphism of schemes $g:Y\to X$ and a generalized effective Cartier divisor $(L,\rho)$ on $X$, $(g^*L, g^*\rho)$ is a generalized effective Cartier divisor on $Y$, where $g^*\rho: g^*L\to g^*\ms O_X=\ms O_Y$. Note that we cannot always pull back effective Cartier divisors, but we can always pull back generalized effective Cartier divisors.
\end{defn}

We define a fibered category $\ms D$ over the category of schemes with objects pairs $(T, (L,\rho))$, where $T$ is a scheme and $(L,\rho)$ is a generalized effective Cartier divisor on $T$. Morphisms in $\ms D$ are pairs 
$$(g,g^b): (T', (L', \rho'))\to (T,(L,\rho)),$$ 
where $g:T'\to T$ is a morphism of schemes and $g^b:(L',\rho')\to (g^*L, g^*\rho)$ is an isomorphism of generalized effective Cartier divisors on $T'$. This fibered category is a stack because line bundles and morphisms of line bundles satisfy effective descent. 

\begin{lem} \cite[10.3.7]{Olssonbook} The stack $\ms D$ is isomorphic to the stack quotient $[\mathbb{A}^1/\mathbb{G}_m]$, where for a scheme $T$, $u\in \mathbb{G}_m(T)$ acts by left multiplication on $t\in \mathbb{A}^1(T)$.
\end{lem}

Consider the morphisms $\mathbb{A}^1\to \mathbb{A}^1$ by $f\mapsto f^n$ and $\mathbb{G}_m\to\mathbb{G}_m$ by $u\mapsto u^n$. This defines a morphism of stacks $p_n:[\mathbb{A}^1/\mathbb{G}_m]\to[\mathbb{A}^1/\mathbb{G}_m]$, which corresponds to the morphism $p_n:\ms D\to \ms D$, $(T,(L,\rho))\mapsto(T,(L^{\otimes n},\rho^{\otimes n}))$. 

 Let $X$ be a scheme, $(L,\rho)$ a generalized effective Cartier divisor on $X$, and consider the fiber product diagram
$$\xymatrix{\ms X_n \ar[r] \ar[d]^-{\pi} & \ms{D} \ar[d]^-{p_n} \\ X \ar[r]^-{(L,\rho)} & \ms{D}},$$
where $(L,\rho):X\to \ms D$ sends an $X$-scheme $f:T\to X$ to $(T,(f^*L,f^*\rho))$.
\begin{defn}
    We define $\ms X_n$ to be the \textit{$n$th root stack of $X$ associated to $(L,\rho)$}, and it is the stack with objects  
    $$(f:T\to X, (M,\lambda), \sigma:(M^{\otimes n}, \lambda^{\otimes n})\cong (f^*L,f^*\rho)),$$
    where $f:T\to X$ is an $X$-scheme, $(M,\lambda)$ is a generalized effective Cartier divisor on $T$, and $\sigma$ is an isomorphism of generalized effective Cartier divisors.
    Morphisms are 
    $$(h,h^b):(f':T'\to X, (M',\lambda'), \sigma')\to (f:T\to X, (M,\lambda), \sigma),$$
    where $h:T'\to T$ is a morphism of $X$-schemes, and $h^b:(M',\lambda')\to (h^*M,h^*\lambda)$ is an isomorphism of generalized effective Cartier divisors on $T'$, such that the diagram
    $$\xymatrix{M'^{\otimes n} \ar[rr]^-{h^{b \otimes n}} \ar[dr]_-{\sigma'} && h^*M^{\otimes n} \ar[dl]^-{h^*\sigma} \\ & f'^*L=h^*f^*L & }$$
commutes.
\end{defn}

Using the Yoneda lemma, the data $(f:T\to X, (M,\lambda), \sigma)$ as an object of $\ms X_n$ corresponds to a morphism $T\to \ms{X}_n$. In particular, the identity map $\ms{X}_n\to \ms X_n$ corresponds to the data 
$(\pi:\ms X_n\to X, (\ms{L}_n, \lambda_n), \sigma),$
where $(\ms L_n,\lambda_n)$ is a generalized effective Cartier divisor on $\ms{X}_n$, whose $n$th power is isomorphic to $(\pi^*L, \pi^*\rho)$. That is, we indeed get a line bundle on $\ms{X}_n$ whose $n$th power is isomorphic to the pullback of $L$. 

When $(L,\rho)=(\ms{O}_X,\cdot f)$, we have an isomorphism \cite[10.3.10(ii)]{Olssonbook}
\begin{equation}
\ms X_n \cong[\uSpec_X(\ms O_X[z]/(z^n-f))/\mu_n],
\end{equation}
where for a scheme $T$, $\mu_n(T)$ acts trivially on $\ms O_X(T)$ and by left multiplication on $z$. We also have an isomorphism of $[\uSpec_X(\ms O_X[z]/(z^n-f))/\mu_n]$ with $X$ over the open set of $X$ where $f$ is invertible, and we get a line bundle on $\ms{X}_n$ whose $n$th power is isomorphic to the pullback of the ideal sheaf on $f$ in $\ms O_X$.

\begin{ex}
Consider the effective Cartier divisor $(x)$ on $\mathbb{A}^1$. We would like to take the $n$th root stack of $\mathbb{A}^1$ with respect to $(x)$. The corresponding generalized effective Cartier divisor is $(I_x, \iota)$, where $\iota:I_x\to \ms{O}_{\mathbb{A}^1}$ is the inclusion of the ideal sheaf at $(x)$. Note that $(I_x,\iota)\cong(\ms{O}_{\mathbb{A}^1}, \cdot x)$ via the diagram
$$\xymatrix{I_x \ar[rr]^-{\cdot 1/x} \ar[dr]_-{\iota} && \ms{O}_{\mathbb{A}^1} \ar[dl]^-{\cdot x} \\ & \ms{O}_{\mathbb{A}^1}}.$$

Then the $n$th root stack of $\mathbb{A}^1$ by the divisor $(x)$ is given by the fiber product diagram 
$$\xymatrix{\ms X_n \ar[r] \ar[d] & \ms{D} \ar[d]^-{p_n} \\ \mathbb{A}^1\ar[r]^-{(\ms O_{\mathbb{A}_1}, \cdot x)} & \ms{D}},$$
and $\ms X_n$ is the stack over $\mathbb{A}^1$ with objects triples
$(f:T\to\mathbb{A}^1, (M,\lambda), \sigma)$ where $(M,\lambda)$ is a generalized effective Cartier divisor on $T$ and 
$\sigma:(M^{\otimes n}, \lambda^{\otimes n})\cong (f^* \ms O_{\mathbb{A}^1}, \cdot f^*x)$
is an isomorphism of generalized effective Cartier divisors. 

We have $\ms X_n$ is isomorphic to $[\mathbb{A}^1/\mu_n]$, where for a scheme $T$, $\zeta\in\mu_n(T)$ acts by left multiplication on $t\in\mathbb{A}^1(T)$. Furthermore, over $\mathbb{A}^1\setminus \{0\}$, the subset of $\mathbb{A}^1$ where $x$ is invertible, $[\mathbb{A}^1/\mu_n]$ is isomorphic to $\mathbb{A}^1$.
\end{ex}

\subsection{} Although Proposition 2.19 is proved in \cite{GS}, we construct the isomorphism in the case where $\ms X$ is a stacky curve. Let $\ms X$ have stacky points $x_1,\ldots, x_r$ with stabilizer groups $G_i=\mu_{n_i}$ and let $\pi:\ms X\to X$ be the coarse space. Let $I_{\pi(x_i)}$ be the ideal sheaf of $X$ at $\pi(x_i)$ and $\ms I_{x_i}$ the ideal sheaf of $\ms X$ at $x_i$. 

Let $\ms X_{n_i}$ be the $n_i$th root stack of $X$ at $\pi(x_i)$, which, as in Example 2.23 is given by the fiber product $X\times_{(\ms O_X,\cdot \pi(x_i)),\ms D, p_n} \ms D$, after identifying the ideal sheaf $(I_{\pi(x_i)}, \iota_{\pi(x_i)})$ with $(\ms O_X,\cdot \pi(x_i))$. Similarly, we identify $(\ms I_{x_i},\iota_{x_i})$ with $(\ms O_{\ms X}, \cdot x_i)$. To give a map $\ms X\to \ms X_{\mathbf{n}}$, we must give a map $\ms X\to \ms X_{n_i}$ for every $i$. We take the map which corresponds to the data $(\pi: \ms X\to X, (\ms O_{\ms X}, \cdot x_i), \sigma)$, where $\sigma: \ms O_{\ms X}^{\otimes n}\to \pi^*\ms O_X\cong \ms O_{\ms X}$ is the multiplication morphism. 

Thus, we get a map $\ms X\to \ms X_{\mathbf{n}}$ which we can check is an isomorphism locally. Over the open set of $X$ away from the points $\pi(x_1), \ldots, \pi(x_r)$, $\ms X$ and $\ms X_{\mathbf{n}}$ are both isomorphic to $X$, and over an open set of $X$ away from all $x_j$ except for one $x_i$, both are isomorphic to $\ms X_{n_i}$ (compare Proposition 2.17 and equation (2.22.1)). Furthermore, on each $\ms X_{n_i}$, we get a generalized effective Cartier divisor $(\ms L_{n_i}, \lambda_{n_i})$ whose $n$th power is isomorphic to $(\pi^*\ms{O}_X=\ms{O}_{\ms{X}_{n_i}}, \pi^*x_i)$, and thus we can conclude that $\ms L_{n_i}\cong \ms I_{x_i}$ on this open set.

\section{Picard Groups of Root Stacks}

In this section, we prove Theorem 1.1. Let $\ms Y$ be a tame stacky curve with trivial generic stabilizer, coarse moduli space $\pi:\ms Y \to X$, and stacky points $x_1, \ldots x_r$ having stabilizer groups $G_i=\mu_{n_i}$. Let $\ms L\in\Pic( \ms Y)$. Then at each $x_i$, we get a line bundle $x_i^*\ms L$ on $BG_i$. 

\begin{prop} \cite[9.E]{Olssonbook}
    Let $R$ be a ring and $G$ a finite group acting on $R$. Then we have an equivalence of categories between quasi-coherent sheaves on $[\Spec R/G]$ and $\mathrm{Mod}_R^G$, the category of $R$-modules with $R$-semi-linear $G$-action. In particular, we get an equivalence of categories between line bundles on BG and rank one representations of $G$. 
\end{prop}
Write $\chi:\Pic({BG})\to \mathrm{Hom}(G, \mathbb{G}_m)$, $\ms{L}\mapsto \chi_{\ms{L}}$ for the equivalence in Proposition 3.1. Now, define 
$$\chi: \Pic(\ms{Y})\to \prod_{i=1}^r \Pic(BG_i)\cong \prod_{i=1}^r \Z/n_i\Z$$
by $\chi(\ms L)=(\chi_{x_1^*(\ms L)}, \ldots, \chi_{x_r^*(\ms L)}),$
where each of the $\chi_{x_i^*(\ms L)}$ is the one dimensional representation of $G_i$ associated to the line bundle at the residual gerbe, $x_i^*(\ms L)$ on $BG_i$. 

We have a sequence
\begin{align}
    \xymatrix{0 \ar[r] & \Pic(X)  \ar[r]^-{\pi^*} & \Pic(\ms{Y}) \ar[r]^-{\chi} & \prod_{i=1}^r\Z/n_i\Z \ar[r] & 0}.
\end{align}

\begin{prop}
    The sequence (3.1.1) is exact.
\end{prop}
\begin{proof}
    The following lemma shows that the sequence is exact at the left and in the middle.

\begin{lem} \cite[6.2]{Picequiv}
Let $S$ be a scheme, $\ms{X}$ a tame stack over $S$, with coarse moduli space $\pi: \ms{X}\to X$. The pullback functor
$$\pi^*:(\text{line bundles on } X)\to (\text{line bundles on } \ms{X})$$
induces an equivalence of categories between line bundles on $X$ and line bundles $\ms{L}$ on $\ms{X}$ such that for all geometric points $x:\Spec k \to \ms{X}$, the representation of $G_{x}$ corresponding to $x^*\ms{L}$ is trivial.
\end{lem}

\begin{rmk} In the case where $\ms{X}=[\Spec R/G]$ and $X=\Spec R^G$ where $R$ is a ring and $G$ is a linearly reductive group scheme, let $\ms L$ be a line bundle on $\ms X$ corresponding to a rank one $R$-module with $R$-semi-linear $G$-action $M$. Then at each geometric point $x:\Spec k\to \ms{X}$ of $\ms{X}$, the representation of $G$ corresponding to $x^*\ms{L}$ is just the action of $G$ on $M\otimes_{R} k$. 
\end{rmk}

By Proposition 3.1, only the finitely many stacky points, $x_1, \ldots x_r$, will have interesting Picard groups at their residual gerbes $BG_i$. Thus, for a line bundle $\ms{L}$ on $\ms{Y}$ to come from a line bundle on $X$, as in Lemma 3.3, it is enough to require that the representations at the finitely many stacky points are trivial. This implies that the cokernel of the pullback map is finite. We now show that the sequence above is exact on the right as well.

Let $\ms L\cong \ms I_{x_i}\otimes \cdots \otimes \ms I_{x_r},$ where $\ms I_{x_i}$ is the ideal sheaf of $x_i$ on $\ms Y$. We will show that at each stacky point, the corresponding representation of $x_i^*\ms L$ is the standard one. Locally at each stacky point, we have the local structure as in the diagram below. 

$$\xymatrix{\ms{U}_i=[\Spec \ms O_{X,\pi(x_i)}[z]/(z^{n_i}-\pi(x_i))/\mu_{n_i}] \ar[r] \ar[d] &\ms{Y} \ar[d] \\ \Spec \ms{O}_{X,\pi(x_i)} \ar[r] & X }$$

Then $\ms{L}|_{\ms{U}_i}$ is given by the tensor product of 1 copy of $\ms I_{x_i}|_{\ms U_i}$ and $r-1$ copies of $\ms O_{\ms U_i}$. Thus, over $\Spec \ms O_{X, \pi(x_i)}$, $\ms L$ is the $n$th root of the pullback along $\pi$ of the ideal sheaf at $\pi(x_i)$. Write $A=\ms O_{X,\pi(x_i)}[z]/(z^{n_i}-\pi(x_i))$. As an $A$-module with $\mu_{n_i}$-action, the pullback of the ideal sheaf at $\pi(x_i)$ is $A\cdot \pi(x_i)$, with $\mu_{n_i}$ acting by multiplication on $z$. Then the $n$th root of this is given by $A\cdot (z)$, with $\mu_{n_i}$ acting by multiplication on $z$. We must check that the representation corresponding to this line bundle pulled back to $BG_i$ is the standard representation. We have $A\cdot (z) \otimes_A k= k\cdot (z)$, with $\mu_n$ still acting by multiplication on $z$, so this indeed corresponds to the standard representation. Therefore, the sequence is exact.
\end{proof}

Let $E$ be the pushout of the diagram 
$$\xymatrix{ e_i \ar[d] &\Z^r \ar[rr]^-{(\cdot n_1,\ldots ,\cdot n_r)} \ar[d]_-{\phi} && \Z^r \ar[d] \\ I_{\pi(x_i)} & \Pic (X) \ar[rr] && E}.$$
That is, $E$ is given by 
\begin{align*}
  E =& \ \mathrm{coker}(\Z^r\to \Z^r\oplus \Pic (X), e_i\mapsto (n_ie_i,-I_{\pi(x_i)})) \\
=&\ (\Z^r\oplus \Pic (X) )/\langle (n_ie_i,-I_{\pi(x_i)})\rangle_{i=1,\ldots, r}
\end{align*}

The vertical map $\Z^r\to E$ is given by $e_i\mapsto (e_i,\ms O_X)$ for all $i$, and the horizontal map $\Pic (X)\to E$ is given by $L\mapsto (0,L)$. 

Consider the following commutative diagram
$$\xymatrix{0 \ar[r] & \Z^r \ar[r]^-{(\cdot n_1,\ldots, \cdot n_r)} \ar[d]^-{\phi}& \Z^r\ar[r] \ar[d] & \prod_{i=1}^r \Z/n_i\Z \ar[r] \ar[d]^-{=} & 0 \\
0 \ar[r] & \Pic (X) \ar[r] \ar[d]^-{=} & E \ar[r] \ar[d] & \prod_{i=1}^r\Z/n_iZ \ar[r]\ar[d]^-{=} & 0 \\
0 \ar[r] & \Pic (X) \ar[r]^-{\pi^*} & \Pic (\ms Y) \ar[r]^-{\chi} & \prod_{i=1}^r\Z/n_i\Z \ar[r] & 0}$$
where the top and bottom sequences are exact, and the map $E\to \Pic (\ms Y)$ is the unique map coming from the universal property of pushout corresponding to the diagram 
$$\xymatrix{e_i \ar[d] & \Z^r \ar[r]^-{(\cdot n_1,\ldots, \cdot n_r)} \ar[d]_-{\phi} & \Z^r \ar[d] & e_i \ar[d] \\
I_{\pi(x_i)} & \Pic (X) \ar[r]^-{\pi^*} & \Pic (\ms Y) & \ms{I}_{x_i}}.$$

We now describe the map $E\to \prod_{i=1}^r \Z/n_i\Z$, by first describing the map $E\to \Pic (\ms Y)$, and then composing this with the map $\chi:\Pic (\ms Y)\to \prod_{i=1}^r \Z/n_i\Z$. 

For each equivalence class $((a_i)_i,L)$ in $E$, write $a_i=n_ib_i+c_i$, where $b_i, c_i$ are integers and $0\leq c_i<n_i$. Then 
$$((a_i)_i,L)=((c_i)_i,L\otimes (\otimes_{i=1}^r I_{\pi(x_i)} ^{\otimes b_i}))=:((c_i)_i, L'),$$
and write
$$((c_i)_i, L')=(0, L')+((c_i)_i, \ms O_X).$$
Since $(0,L')$ must map to $\pi^*L'\in\Pic (\ms Y)$, and $((c_i)_i, \ms O_X)$ must map to $\otimes_{i=1}^r \ms I_{x_i}^{\otimes c_i}\in \Pic (\ms Y)$, we have 
$$((a_i)_i, L)\mapsto \pi^*L'\otimes (\otimes_{i=1}^r \ms I_{x_i}^{\otimes c_i}).$$
Then the map $E\to \prod_{i=1}^r \Z/n_i\Z$ is given by 
$$((a_i)_i, L)\mapsto (c_i)_i.$$
We can now check directly that the middle sequence is also exact, so by the 5-lemma, $E\cong \Pic (\ms Y)$. This completes the proof of Theorem 1.1.
\qed

\section{Interlude on Gerbes}
In this section, we review gerbes, following \cite[IV]{Giraudbook} and \cite[12.2]{Olssonbook}. In particular, a nonabelian gerbe is associated not to a sheaf of groups, but rather, a band. We discuss the stack of bands in this section as well. 

\begin{defn}
    Let $C$ be a site. A \textit{gerbe over $C$} is a stack $\pi: \mathscr{G}\to C$ such that
    \begin{enumerate}
    \item[(G1)] For any $Y\in C$, there exists a covering $\{Y_i\to Y\}_{i\in I}$ such that $\ms{G}(Y_i)$ is nonempty for each $i\in I$.
    \item[(G2)] For any two objects $y,y'\in \ms{G}(Y)$, there exists a covering $\{f_i:Y_i\to Y\}_{i\in I}$ such that $f_i^*y$ and $f_i^*y'$ are isomorphic in $\ms{G}(Y_i)$ for all $i\in I$. 
    \end{enumerate}
\end{defn}

For any gerbe $\pi:\ms{G}\to C$ and $Y\in C$, if $\sigma:y\to y'$ is an isomorphism in $\ms{G}(Y)$, then we get an induced isomorphism 
$$\Phi:\uAut_Y(y)\to \uAut_Y(y'),\  \alpha\mapsto \sigma \alpha \sigma^{-1}.$$

\begin{rmk}
    To be more precise, $\uAut_Y(y)$ is the sheaf of groups on $C/Y$, given by 
    $$(f:X\to Y)\mapsto \Aut_X(f^*y).$$
    Then the isomorphism $\sigma:y\to y'$ induces an isomorphism $\Phi:\uAut_Y(y)\to \uAut_Y(y')$, which, over $f:X\to Y$ is the isomorphism 
    $$\Aut_X(f^*y)\to \Aut_X(f^*y'), \ \alpha\mapsto (f^*\sigma)\alpha(f^*\sigma)^{-1},$$
    where $f^*\sigma:f^*y\to f^*y'$ is an isomorphism in $\ms{G}(X)$. 
\end{rmk}

\begin{rmk}
    If $\ms Y$ is a stack over a scheme $S$, then by a gerbe over $\ms Y$, we mean a gerbe over the \'etale site of $\ms Y$, $p:\ms X\to \ms Y$. Usually, by $\uAut_T(t)$ for an $S$-scheme $T$ and a $T$-point $t:T\to \ms X$, we mean the fiber product $T\times_{\ms X} \ms I_{\ms X/ S}$, but to apply the definitions in this section, we want $\uAut_T(t)$ to be the fiber product $T\times_{\ms X} \ms I_{\ms X/\ms Y}$. If confusion arises, we will specify $\uAut_T(t)$ to be viewed in $\ms I_{\ms X/\ms Y}$ or $\ms I_{\ms X/S}$.
\end{rmk}
The isomorphism $\Phi$ is not canonical, but depends on the choice of isomorphism $\sigma$. Indeed, given a different choice $\sigma':y\to y'$, the map $\Phi$ is given by $\alpha\mapsto \sigma'\alpha\sigma'^{-1}$, which differs from the first map by conjugation by $\sigma\sigma'^{-1}\in\uAut_Y(y)$. In the case where $\uAut_Y(y)$ is abelian for every $y\in \ms{G}$, the isomorphism $\Phi$ is canonical. We get a sheaf of abelian groups $A$ on $C$ and a canonical isomorphism $A_{\pi(y)}\to \uAut_Y(y)$ for every $y\in \ms{G}$. In the abelian case, we have the following definition.

\begin{defn}
    Let $C$ be a site and $A$ a sheaf of abelian groups on $C$. An \textit{$A$-gerbe over $C$} is a gerbe over $C$, $\pi:\ms{G}\to C$ together with an isomorphism of sheaves of groups on $C/\pi(y)$,
    $$\iota_y:A_{\pi(y)}\to \uAut_{\pi(y)}(y),$$
    for every $y\in \ms{G}$ such that 
    \begin{enumerate}
    \item[(G3)] For every $Y\in C$ and isomorphism $\sigma: y\to y'$ in $\ms{G}(Y)$, the diagram 
    $$\xymatrix{ & A \ar[dl]_-{\iota_y} \ar[dr]^-{\iota_{y'}} & \\ \uAut_Y(y) \ar[rr]^-{\sigma} & & \uAut_Y(y')}$$
    commutes. 
    \end{enumerate}
\end{defn}

In the case where not every $\uAut_Y(y)$ is abelian, we get a band $G$ on $C$ and a canonical isomorphism $\iota_y: G_{\pi(y)}\to \underline{\uAut_Y(y)}$ for every $y\in\ms G$, where $\underline{\uAut_Y(y)}$ is the band on $C/Y$ associated to $\uAut_Y(y)$. We say that the gerbe $p:\ms{G}\to C$ is \textit{banded by $G$}. We describe the category of bands now. 

Let $\sGp$ be the category of sheaves of groups on a site $C$, and let $\ms F$ and $\ms G$ be objects in $\sGp$. Define 
$$\Hex(\ms F,\ms G):=\ms G\backslash \Hom(\ms F,\ms G) /\ms F$$
where for $\varphi\in\Hom(\ms F,\ms G)$, $g\in G$, $f\in F$, $x\in F$, the action of $G$ on the left and $F$ on the right on $\Hom(F,G)$ is given by
$$(g\cdot \varphi)(x)=g\varphi(x) g^{-1},$$
$$(\varphi\cdot f)(x)= \varphi(fxf^{-1}).$$

\begin{lem}
    We have isomorphisms
    $$\ms G\backslash \Hom(\ms F,\ms G)/\ms F\cong \ms G\backslash \Hom(\ms F,\ms G)\cong \Int(\ms G)\backslash \Hom(\ms F,\ms G),$$ where $\Int(\ms G)=\ms G/Z(\ms G)$.
\end{lem}

\begin{proof}
    Clearly $\ms G\backslash \Hom(\ms F,\ms G)$ surjects onto $\ms G\backslash \Hom(\ms F,\ms G)/\ms F$. Let $\varphi,\varphi'\in\Hom(\ms F,\ms G)$ such that for some $f\in \ms F, g\in \ms G$, and all $x\in \ms F$, $\varphi'(x)= g\varphi(fxf^{-1}) g^{-1}$. Then 
    $$\varphi'(x)= (g\varphi(f))\varphi(x)(g\varphi(f))^{-1},$$
    so $\varphi$ and $\varphi'$ represent the same class in $\ms G\backslash \Hom(\ms F,\ms G)$. Finally, for $g\in Z(\ms G)$, and any $\varphi\in\Hom(\ms F,\ms G)$, $g\cdot \varphi=\varphi$, giving the second isomorphism. 
\end{proof}

\begin{rmk}
    In particular, given two sheaves of groups $\ms F$ and $\ms G$ with $\ms G$ abelian, $\Hex(\ms F,\ms G)\cong \Hom(\ms F,\ms G)$.
\end{rmk}

Define the category of pre-bands on $C$, $\Ba^{\mathrm{pre}}$, to be the category whose objects are sheaves of groups on $C$, and for $\ms F,\ms G\in \mathrm{Ob}(\Ba^{\mathrm{pre}})$, 
$$\Hom_{\Ba^{\mathrm{pre}}}(\ms F,\ms G)=\Hex(\ms F,\ms G).$$
We have a natural functor 
$$\Pi: \sGp\to \Ba^{\mathrm{pre}}$$
sending a sheaf of groups $\ms G$ to itself, and a morphism $\varphi:\ms F\to \ms G$ to its class in $\Hex(\ms F,\ms G)$. This functor is essentially surjective and full, but not faithful. We will write the image of $\ms G\in\sGp$ under $\Pi$ by $\underline{\ms G}$. Define the category of bands on $C$, $\Ba$, to be the stackification of the category of pre-bands on $C$. We have a functor
\begin{align}
    \Pi':\sGp\to \Ba
\end{align}
which is given by $\Pi$ composed with the stackification functor. This functor is no longer essentially surjective, but is locally essentially surjective. Bands can be constructed locally and satisfy decent, and a band can locally be lifted to a sheaf of groups.



With this discussion of bands, we can make the following definition. \begin{defn}
    Let $C$ be a site and $H$ a band on $C$. A \textit{gerbe banded by $H$ over $C$} is a gerbe over $C$, $\pi:\ms{G}\to C$ together with an isomorphism of bands on $C/\pi(y)$,
    $$\iota_y:H_{\pi(y)}\to \underline{\uAut_{\pi(y)}(y)},$$
    for every $y\in \ms{G}$ such that 
    \begin{enumerate}
    \item[(G3')] For every $Y\in C$ and isomorphism $\sigma: y\to y'$ in $\ms{G}(Y)$, the diagram  
    $$\xymatrix{ & H \ar[dl]_-{\iota_y} \ar[dr]^-{\iota_{y'}} & \\ \underline{\uAut_Y(y)} \ar[rr]^-{\sigma} & & \underline{\uAut_Y(y')}}$$
    commutes.
    \end{enumerate}
\end{defn}

\begin{defn}
    Let $A$ be a sheaf of abelian groups on $C$. A \textit{morphism of $A$-gerbes} $f:(\pi':\ms{G}'\to C, \{\iota_{x'}\})\to(\pi:\ms{G}\to C,\{\iota_x\})$ is a morphism of stacks $f:\ms{G}'\to \ms{G}$ such that for every $x'\in\ms{G}'$, the diagram
    $$\xymatrix{ & A \ar[dl]_-{\iota_{x'}} \ar[dr]^-{\iota_{f(x')}} & \\ \uAut_{\pi'(x')}(x') \ar[rr]^-{f_*} & & \uAut_{\pi(f(x'))}(f(x'))} $$
    commutes. Here $f_*$ is the natural map on morphisms in $\ms{G}'$ coming from the morphism of stacks. 
\end{defn}

\begin{defn}
    Let $u:A\to B$ be morphism of sheaves of abelian groups on $C$. A \textit{$u$-morphism of gerbes} $f:(\pi':\ms{G}'\to C, \{\iota_{x'}\})\to(\pi:\ms{G}\to C,\{\iota_x\})$ from an $A$-gerbe $\ms{G}'$ to a $B$-gerbe $\ms{G}$ is a morphism of stacks $f:\ms{G}'\to \ms{G}$ with morphism of stabilizers given by $u$, that is, such that for every $x'\in\ms{G}'$, the diagram 
    $$\xymatrix{A\ar[r]^-{u} \ar[d]_-{\iota_{x'}} & B\ar[d]^-{\iota_{f(x')}} \\ \uAut_{\pi'(x')}(x') \ar[r]^-{f_*} & \uAut_{\pi(f(x'))}(f(x'))}$$
    commutes.
\end{defn}

\begin{rmk} 
    We have analogous definitions for a morphism of gerbes banded by a  band $H$, and a $u$-morphism of gerbes where $u:G\to H$ is a morphism of bands, by requiring the diagrams to commute after passing the category of bands.
\end{rmk}

\begin{thm} \cite[IV.3.4.2]{Giraudbook}, \cite[12.2.8]{Olssonbook}
Let $A$ be a sheaf of abelian groups on a site $C$. We have a bijection
$$H^2(C,A)\to \{ \mathrm{isomorphism} \ \mathrm{classes} \ \mathrm{of} \ A\mathrm{-gerbes}\}.$$
\end{thm}

\begin{defn}
    Let $G$ be a band on $C$, and define $H^2(C,G)$ to be the set of isomorphism classes of gerbes banded by $G$ on $C$. 
\end{defn}

For a morphism of sheaves of abelian groups $A\to B$, we have a map 
$$H^2(C,A)\to H^2(C,B)$$ 
given by the usual cohomology map, but we would like to have a map like this for nonabelian gerbes too. Let $u:L\to M$ be a morphism of bands on $C$. Then we define a correspondence between $H^2(C,L)$ and $H^2(C,M)$ \cite[IV.3.1.4]{Giraudbook}, where $p\in H^2(C,L)$ corresponds to $q\in H^2(C,M)$ if there is a gerbe $P$ in the class of $p$ and $Q$ in the class of $q$ and a $u$-morphism $m:P\to Q$. It follows from the definition that the correspondence is transitive for two morphisms of bands $u:L\to M$, $v:M\to N$. 

The correspondence that Giraud defines does not always result in a map, but if $M$ is abelian, then we get a morphism 
$$u^{(2)}:H^2(C,L)\to H^2(C,M),$$
defined by the correspondence above \cite[IV.3.1.5, 3.1.8]{Giraudbook}.

\begin{prop} \cite[12.F(v)]{Olssonbook}
    Let $u:A\to B$ be a morphism of sheaves of abelian groups on a site $C$. Then the map $u^{(2)}: H^2(C,A)\to H^2(C,B)$ defined by the correspondence agrees with the cohomology morphism $u_*:H^2(C,A)\to H^2(C,B)$.
\end{prop}

In accordance with Proposition 4.13, we will write $u_*$ for the correspondence $u^{(2)}$ whenever it is a morphism. 

\section{Rigidification}

In this section, we discuss rigidification following \cite[A]{Tame}, in which a subgroup of the inertia can be removed from certain stacks, so that we can realize these stacks as gerbes over their rigidifications. Let $S$ be a scheme or an algebraic space, and let $\ms X\to S$ be a locally finitely presented algebraic stack. Write $\ms I$ for the inertia stack $\ms I (\ms X)$. For an $S$-scheme $T$ and $t\in \ms X(T)$, $\uAut_T(t)= T\times_{\ms X}\ms{I}$ is a locally finitely presented group algebraic space.  
Furthermore, for a morphism $t'\to t$ over an $S$-morphism $T'\to T$, we have an isomorphism $\uAut_{T'}(t') \cong T'\times_T \uAut_T(t)$. 

Let $G\subseteq \ms{I}$ be a flat finitely presented subgroup stack. The pullback of $G$ to $T$ by $t$ is a flat finitely presented subgroup algebraic space $G_t\subseteq \uAut_T(t)$, and again $G_{t'}\cong T'\times_T G_t$ for any morphism $t'\to t$ over any $S$-morphism $T'\to T$. 

Conversely, given a subgroup algebraic space $G_t\subseteq \uAut_T(t)$ for each object $t:T\to \ms X$ which is flat and finitely presented over $T$ such that for each morphism $t'\to t$ over any $S$-morphism $T'\to T$, the isomorphism $\uAut_{T'}(t') \cong T'\times_T \uAut_T(t)$ carries $G_{t'}$ to $T'\times_T G_t$, we get a unique flat finitely presented subgroup stack $G\subseteq \ms I$ such that for any $t:T\to \ms X$, the pullback of $G$ to $T$ by $t$ is $G_t$. This condition implies that each subgroup $G_t$ is normal in $\uAut_T(t)$. Throughout this section, the automorphism groups $\uAut_T(t)$ will be viewed in the full inertia groups $\ms I_{\ms X/S}$ or $\ms I_{\ms Y/S}$ (see Remark 4.3).

\begin{thm} \cite[A.1]{Tame}
    Let $S$ be a scheme or algebraic space, and let $\ms X\to S$ be a locally finitely presented algebraic stack with inertia stack $\ms I$. Let $G$ be a flat finitely presented subgroup stack of $\ms I$. Then there exists a locally finitely presented algebraic stack $\ms X \thickslash G$ over $S$ and a morphism $\pi: \ms X\to \ms X\thickslash G$ such that $\ms X$ is an fppf-gerbe over $\ms X\thickslash G$ banded by $\underline{G}$ and for all $S$-schemes $T$ and $t\in \ms X(T)$, the natural morphism
    $$\uAut_T(t)\to \uAut_T(\pi(t))$$
    is surjective with kernel $G_t$.
    Furthermore, if $G$ is finite over $\ms X$, then $\pi$ is proper, and if $G$ is \'etale, then $\pi$ is \'etale. 
\end{thm}
\begin{rmk}
    The band $\underline{G}$ on $\ms X\thickslash G$ is locally given by the following procedure. For any section $t: T\to \ms X\thickslash G$, there exists a covering $s_i:T_i\to T$ such that $s_i^*t$ lifts to $\ms X$. We would like $\underline{G}$ on $s_i^*t$ to be given by the group $G$ on the lift of $s_i^*t$. These groups only glue together canonically up to conjugation on $\ms X\thickslash G$, so we get a band on $\ms X\thickslash G$. Note that the pullback of $\underline{G}$ by $\pi$ agrees with $\Pi'(G)=\underline{G}$, (4.6.1), as bands on $\ms X$.
\end{rmk}

\begin{prop}
    Let $S$ be a scheme or algebraic space, $\ms X\to S$ a locally finitely presented algebraic stack with inertia stack $\ms I$, and $G$ a flat finitely presented subgroup stack of $\ms I$. Suppose $\ms Y$ is an algebraic stack over $S$ with a morphism $p:\ms X\to \ms Y$ such that $\ms X$ is an fppf-gerbe over $\ms Y$  and for all $S$-schemes $T$ and $x\in \ms X(T)$, the natural morphism
    $$\uAut_T(x)\to \uAut_T(p(x))$$
    is surjective with kernel $G_x$. Then $\ms Y$ is equivalent to $\ms X\thickslash G$, and thus $p$ is banded by $\underline{G}$ and $\ms Y$ is locally finitely presented over $S$. 
\end{prop}
\begin{proof}
      As in the proof of Theorem 5.1 \cite[A.1]{Tame}, define $(\ms{X}\thickslash G) ^{\mathrm{pre}}$ to be the stack with objects of $\ms{X}$, and for any $S$-scheme $T$, and $x, x'\in \ms{X}(T)$, 
      $$\Hom_{(\ms{X}\thickslash G) ^{\mathrm{pre}}} (x',x)=G_x\backslash\Hom_\ms{X}(x',x).$$
      Let $\pi:\ms{X}\to (\ms{X}\thickslash G) ^{\mathrm{pre}}$ be the natural morphism of prestacks which is the identity map on objects, and is the natural quotient map on morphisms. We define a morphism of prestacks $\phi:(\ms{X}\thickslash G)^{\mathrm{pre}}\to \ms{Y}$, which on objects is given by $x\mapsto p(x)$. For morphisms, we have 
      \begin{equation}
      \Hom_{\ms X}(x',x)\to \Hom_{\ms Y}(p(x'),p(x))
      \end{equation}
      defined by the morphism of stacks $p:\ms X\to \ms Y$, which we would like to show descends to a morphism from the quotient $G_x\backslash\Hom_\ms{X}(x',x)$.

      If $\Hom_{\ms{X}}(x',x)$ is nonempty, then $\Hom_{\ms Y} (p(x'),p(x))$ is also nonempty, and we have $\Hom_{\ms X}(x',x)\cong \Aut(x)$, and $\Hom_{\ms Y} (p(x'),p(x))\cong \Aut(p(x)) $. By assumption, the morphism $\uAut_T(x)\to \uAut_T(p(x))$ is surjective with kernel $G_x$, so the map (5.3.1) indeed descends to the quotient by $G_x$, which defines
       $$\Hom_{(\ms{X}\thickslash G) ^{\mathrm{pre}}} (x',x)\to \Hom_{\ms Y}(p(x'),p(x)).$$
       Moreover, this is an isomorphism.

       We also have a natural morphism $$\iota:(\ms{X}\thickslash G) ^{\mathrm{pre}}\to \ms{X}\thickslash G, $$
       the stackification morphism, and thus we get a unique morphism $\psi: \ms{X}\thickslash G \to \ms{Y}$ making the diagram of prestacks
       $$\xymatrix{ \ms{X} \ar[r]^-{p} \ar[d]_-{\pi} & \ms Y \\ (\ms X \thickslash G)^{\mathrm{pre}} \ar[ur]^{\phi} \ar[r]^-{\iota} & \ms X \thickslash G \ar[u]_-{\psi} }$$
       commute. We show that $\psi:\ms{X}\thickslash G \to \ms Y$ is an isomorphism of stacks.

       First, we show that $\psi$ is fully faithful, so we check that we have an isomorphism of sheaves
       $$\uIsom_{\ms{X} \thickslash G} (x',x)\cong \uIsom_{\ms Y} (p(x'),p(x)),$$
       which we can check locally on stalks, using the prestack $(\ms X\thickslash G)^{\mathrm{pre}}$. 
       Since 
       $$\uIsom_{\ms X\thickslash G}(x',x)= (\uIsom_{(\ms X\thickslash G) ^{\mathrm{pre}}}(x',x))^{\mathrm{sh}},$$
       where sh denotes sheafification, and $\Isom_{(\ms X\thickslash G)^{\mathrm{pre}}}(x',x)\cong \Isom_{\ms Y}(p(x'),p(x))$,
       we have the required isomorphism. 

       For essential surjectivity, since $p:\ms X\to \ms Y$ is a gerbe, for all $S$-schemes $T$, and $y\in \ms Y(T)$, there exists a cover $\{f_i:T_i\to T\}$ such that $f_i^*y\in \ms Y(T_i)$ lifts to $\ms X(T_i)$. The image of this lift in $\ms X\thickslash G$ maps to $f_i^*y$ under $\psi$. 
\end{proof}

\begin{prop}
    Let $S$ be a scheme or algebraic space, $\ms Y$ a locally finitely presented algebraic stack over $S$, $G$ a flat finitely presented band on $\ms Y$, and $p:\ms X\to \ms Y$ a gerbe banded by $G$. Then $\ms X$ is locally finitely presented over $S$ and there exists a flat finitely presented sheaf of groups $G'\subseteq \ms I(\ms X)$ such that $\ms X \thickslash G'\cong \ms Y$, and $\underline{G'}\cong G$ is an isomorphism of bands on $\ms Y$.
\end{prop}
\begin{proof}
    To check that $\ms X$ is locally finitely presented over $S$, we may assume that $p:B_{\ms Y} G\to \ms Y$ is the trivial gerbe, so has a section $s: \ms Y\to B_{\ms Y} G$, which is a covering of $B_{\ms Y} G$. Consider the fiber product diagram
    $$\xymatrix{G\ar[r] \ar[d] & \ms Y \ar[d]\\ \ms Y \ar[r] & B_{\ms Y} G }.$$
    Then since $G\to \ms Y$ is finitely presented, $\ms Y\to B_{\ms Y} G$ is locally finitely presented, so $B_{\ms Y}G$ has a locally finitely presented covering by a locally finitely presented algebraic stack, so $B_{\ms Y}G$ is also a locally finitely presented algebraic stack. 

    Consider the commutative diagram
    $$\xymatrix{\ms I (\ms X) \ar[r] \ar[d] & \ms I (\ms Y) \ar[d] \\ \ms X \ar[r]^-{p} & \ms Y },$$
    which maps $(x,g)\in \ms I (\ms X)$ to $x\in \ms X$ and similarly for $\ms I(\ms Y) \to \ms Y$, and sends $(x,g)\in \ms I (\ms X)$ to $(p(x), p(g))\in \ms I (\ms Y) $. Here, we are using the characterization of the inertia stack of $\ms X$ as having objects pairs $(x,g)$ with $x\in\ms X$ and $g\in \Aut(x)$ \cite[8.1.18]{Olssonbook}. This gives us a morphism of algebraic stacks 
    $$\ms I (\ms X) \to \ms I (\ms Y) \times_{\ms Y} \ms X,$$ where
    $$(x,g)\mapsto((p(x), p(g)), x, \mathrm{id}: p(x)\to p(x)).$$ 
    We define a sheaf of groups $G'\subseteq \ms I (\ms X)$ to be the kernel of this morphism. The zero section of the inertia stack consists of elements $(x,e)$, where $e$ is the identity automorphism of $x$. The zero section of the stack $\ms I (\ms Y) \times_{\ms Y} \ms X$ consists of elements $((y,e),x,\sigma: y\to p(x))$. Thus, the kernel $G'$ consists of elements $(x,g)$ such that $p(g)$ is the identity automorphism of $p(x)$. 

    Recall that for an $S$-scheme $T$ and a $T$-point $x\in \ms X(T)$, we have a cartesian diagram 
    $$\xymatrix{\uAut_T(x) \ar[r] \ar[d] & \ms I(\ms X) \ar[d] \\ T\ar[r]^-{x} & \ms{X} },$$
    which can be seen by the fact that every element of the fiber product $((x',g),\sigma:x'\cong x)$ is isomorphic to $((x,g'),\mathrm{id}:x\cong x)$ via the isomorphism $\sigma$, where $g'=\sigma^{-1}g\sigma$. 

    Similarly, we have a cartesian diagram 
    $$\xymatrix{\uAut_T(p(x)) \ar[r] \ar[d] & \ms I(\ms Y) \times_{\ms Y} \ms X \ar[d]\\ T \ar[r]^-{x} & \ms X },$$

    which can be seen by the fact the every element of the fiber product 
    $$((y,h),x',\sigma: y\cong p(x'), \epsilon: x'\cong x)$$ is isomorphic to 
    $$((p(x),h'), x, \mathrm{id}:p(x)\cong p(x), \mathrm{id}:x\cong x)$$
    via the isomorphism $(p(\epsilon)\circ \sigma, \epsilon)$, where $h'= p(\epsilon)\sigma h (p(\epsilon)\sigma)^{-1}$. 

    Thus, we have the following diagram with two cartesian squares.
    $$\xymatrix{\uAut_T(x)\ar[rr] \ar[dr] \ar[dd] && \ms I (\ms X) \ar[dd] \ar[dr] & \\ & \uAut_T(p(x)) \ar[rr] \ar[dl] && \ms I (\ms Y) \times_{\ms Y} \ms X \ar[dl] \\ T \ar[rr]^-{x} && \ms X &}$$

    Then $G'_x$ is the kernel of the morphism 
    $\uAut_T(x)\to \uAut_T(p(x)).$
    We can check that this morphism is surjective locally, on geometric points. For a geometric point $x:\Spec k\to \ms X$, the following lemma tells us that $\Aut(x)\to \Aut(p(x))$ is surjective if and only if $B\Aut(x)\to B\Aut(p(x))$ is a gerbe, which is true as $p:\ms X\to \ms Y$ is locally $B\Aut(x)\to B\Aut(p(x))$.

    \begin{lem}
        Let $\alpha: G\to H$ be a morphism of groups. The induced map $BG\to BH$, with $BG=[\Spec k/G]$, $BH=[\Spec k/H]$, and $k$ an algebraically closed field, is a gerbe if and only if $\alpha$ is surjective.
    \end{lem}
    \begin{proof}
        Let $x:\Spec k\to BH$ be the covering corresponding to the trivial $H$-torsor. Then the fiber product $\ms P$ of the diagram
        $$\xymatrix{\ms P \ar[r] \ar[d] & BG \ar[d] \\ \Spec k \ar[r]^-{x} & BH}$$
        is the stack with objects pairs $(P,\sigma)$ where $P$ is a $G$-torsor and $\sigma:\alpha_*P\cong H$ is an isomorphism of $H$-torsors. Two objects $(P_1,\sigma_1)$ and $(P_2, \sigma_2)$ are isomorphic if there exists an isomorphism of $G$-torsors $\lambda:P_1\cong P_2$ such that the diagram 
        $$\xymatrix{\alpha_*P_1 \ar[rr]^-{\alpha_*\lambda} \ar[dr]_-{\sigma_1} && \alpha_* P_2 \ar[dl]^-{\sigma_2} \\ & H & }$$
        commutes. 

        If $\alpha$ is surjective, then $BG\to BH$ is a $\ker(\alpha)$-gerbe. Conversely, if $BG\to BH$ is a gerbe, then $\ms P$ must consist of only one isomorphism class of objects, that is, any two objects of $\ms P$ are isomorphic. Let $(G, \mathrm{id}:\alpha_*G=H\cong H)$ and $(G, \cdot h: \alpha_*G=H\cong H)$ be two objects of $\ms P$. Any element $g$ of $G$ gives an isomorphism $\lambda=\cdot g:G\cong G$, but the diagram
        $$\xymatrix{\alpha_*G\ar[rr]^-{\alpha_*(\lambda)=\cdot \alpha(g)} \ar[dr]_-{\mathrm{id}} && \alpha_*G\ar[dl]^-{\cdot h} \\ & H & }$$
        commutes if and only if $\alpha(g)=h^{-1}$. Thus, there is only one isomorphism class if and only if $\alpha$ is surjective. 
    \end{proof}

    Furthermore, the lemma tells us that $B\Aut(x)\to B\Aut(p(x))$ is a gerbe banded by $\ker(\Aut(x)\to \Aut(p(x)))=G'_x$, so each $G'_x$ agrees with the group coming from the constant band $G_x$, and thus also $G'$ is flat and finitely presented. Then by Proposition 5.3, $\ms Y \cong \ms X\thickslash G'$, so that $p:\ms X\to \ms Y$ is a gerbe banded by $\underline{G'}$, and thus $\underline{G'}=G$. 
\end{proof}

\begin{cor}
    Let $S$ be a scheme or algebraic space, and let $\ms X\to S$ be a locally finitely presented algebraic stack with inertia stack $\ms I$. Let $H\subseteq G$ be flat finitely presented subgroup stacks of $\ms I$. Then $G/H$ is a flat finitely presented subgroup stack of $\ms X\thickslash H$, and $(\ms X \thickslash H)\thickslash (G/H)\cong \ms X\thickslash G$, and we have a commutative diagram
    $$\xymatrix{\ms X \ar[r]^-{\pi_H} \ar[dr]_-{\pi_G} & \ms X\thickslash H \ar[d]^-{\pi_{G/H}} \\ & \ms X\thickslash G}.$$
\end{cor}
\begin{proof}
    By Theorem 5.1, we get two locally finitely presented algebraic stacks $\ms X\thickslash G$ and $\ms X \thickslash H$ and morphisms $\pi_G:\ms X \to \ms X \thickslash G$ and $\pi_H:\ms X \to \ms X \thickslash H$ such that $\pi_G$ and $\pi_H$ are gerbes banded by $\underline{G}$ and $\underline{H}$ respectively, and for all $S$-schemes $T$ and $T$-points $t\in \ms X(T)$, $\uAut_T(t)/G_t\cong \uAut_T(\pi_G(t))$ and $\uAut_T(t)/H_t\cong \uAut_T(\pi_H(t))$. 
    
    Since $G_t$ and $H_t$ are normal flat finitely presented subgroup algebraic spaces of $\uAut_T(t)$ over $T$, $G_t/H_t$ is a flat finitely presented subgroup algebraic space of $\uAut_T(t)/H_t\\ \cong \uAut_T(\pi_H(t))$ over $T$. Indeed, we can check that $G_t/H_t$ is flat and finitely presented locally, and the cokernel of a morphism of finitely presented modules is finitely presented, and any quotient of a flat module is flat.
    
    Moreover, the construction of the $G_t/H_t$ satisfies the condition that for any morphism $t'\to t$ over an $S$-morphism $T'\to T$, the isomorphism
    $$\uAut_{T'}(\pi_H(t'))\cong T'\times_T \uAut_T(\pi_H(t))$$
    carries $G_{t'}/H_{t'}$ to $T'\times_T G_t/H_t$. Note that we only need to check this condition and that $G_t/H_t$ is flat finitely presented on points $t$ of $\ms X$ because every point and morphism of $\ms X \thickslash H$ locally lifts to $\ms X$. Thus, we get a unique flat finitely presented subgroup stack of $\ms I(\ms X \thickslash H)$ which must be $G/H$. Thus, again by Theorem 5.1, we have a locally finitely presented algebraic stack $(\ms X \thickslash H)\thickslash (G/H)$ over $S$ and a morphism $\pi_{G/H}: \ms X \thickslash H \to (\ms X \thickslash H)\thickslash (G/H)$ which is an fppf-gerbe banded by $\underline{G/H}$, such that for every $S$-scheme $T$ and $s\in (\ms X\thickslash H)(T)$, the morphism $\uAut_T(s)\to \uAut_T(\pi_{G/H}(s))$ is surjective with kernel $(G/H)_s$.

    Clearly, the composition 
    $$\pi_{G/H}\circ \pi_H: \ms X\to \ms X \thickslash H \to (\ms X\thickslash H)\thickslash (G/H)$$
    is an fppf-gerbe, and the composition
    $$\uAut_T(t)\to \uAut_T(\pi_H(t))\to \uAut_T(\pi_{G/H}(\pi_H(t)))$$
    is surjective with kernel $G_t$. Therefore, by Proposition 5.3, $(\ms X \thickslash H)\thickslash (G/H)\cong \ms X\thickslash G$.
\end{proof}

\section{Picard groups of stacky curves: the gerbe case}

In this section, we prove Theorem 1.2. As remarked in the introduction, Theorem 1.2 holds more generally than stated, so let $S$ be a locally noetherian scheme, and let $\ms{Y}$ be a locally finitely presented algebraic stack over $S$. Let $\ms{H}$ be a tame locally constant finitely presented band on $\ms{Y}$. Let $\pi:\ms{X}\to \ms{Y}$ be a gerbe banded by $\ms{H}$, and let $[\ms{X}]\in H^2(\ms{Y},\ms H)$ be the associated cohomology class. 

We define a map $\chi: \Pic (\ms X)\to \Hom(\ms H, \mathbb{G}_m)$ as follows. 
Let $L$ be a line bundle on $\ms{X}$. For each $S$-scheme $T$ and $T$-point $t:T\to\ms{X}$, we get a line bundle $t^*L$ on $T$, and an isomorphism $\chi_L(a):t^*L\to t^*L$ for each $a\in \uAut_T(t)$, now viewing $\uAut_T(t)$ inside $\ms I_{\ms X/\ms Y}$, see Remark 4.3. Since $\ms{X}$ is a gerbe banded by $\ms H$, $\ms H_T\cong \underline{\uAut_{T}(t)}$ as bands. Thus, the line bundle $L$ defines a character $(\chi_L: \underline{\uAut_{T}(t)}\to \mathrm{GL}(t^*L))\in \Hom(\ms H_T, \mathbb{G}_m)$. 
As described in Section 4, these morphisms glue together only as morphisms of bands, and we get a morphism of bands $\ms H\to \mathbb{G}_m$. 

Note that for two line bundles $L$ and $M$ on $\ms{G}$, the character corresponding to $L\otimes M$ is 
$$\chi_{L\otimes M}(a):L\otimes M\to L\otimes M$$
by $l\otimes m\mapsto \chi_L(a)(l)\otimes \chi_M(a)(m)=\chi_L\cdot \chi_M (a)(l\otimes m)$.
Then we have a homomorphism $\chi:\Pic(\ms{G})\to \Hom(\ms H,\mathbb{G}_m)$.

Next, we define a map $(-)_*[\ms X]: \Hom(\ms H, \mathbb{G}_m)\to H^2(\ms Y, \mathbb{G}_m)$ by 
$$(\epsilon: \ms H \to \mathbb{G}_m)\mapsto \epsilon_*([\ms X]),$$
as described in Section 4, with $\epsilon_*$ given by Giraud's correspondence relation $\epsilon^{(2)}$ \cite[IV.3.1.4]{Giraudbook}. We will describe this map more explicitly here. 
Let $\epsilon: \ms H\to \mathbb{G}_m$ be a morphism of bands, and let $\phi: \ms H \to \ms H^{\mathrm{ab}}$ be the morphism of bands corresponding the abelianization morphism of groups (see Remark 6.1 below). Recall that since $\ms H^{\mathrm{ab}}$ and $\mathbb{G}_m$ are both sheaves of abelian groups, any morphism of bands $\ms H\to \mathbb{G}_m$, or $\ms H\to \ms H^{\mathrm{ab}}$ corresponds to a unique such morphism of groups. Thus, there exists a unique morphism of groups $\epsilon^{\mathrm{ab}}: \ms H^{\mathrm{ab}}\to \mathbb{G}_m$ such that $\epsilon= \epsilon^{\mathrm{ab}} \circ \phi$. Then $\epsilon_*: H^2(\ms Y, \ms H)\to H^2(\ms Y, \mathbb{G}_m)$ is the composition of 
$\phi_*:H^2(\ms Y, \ms H)\to H^2(\ms Y, \ms H^{\mathrm{ab}})$ and $\epsilon^{\mathrm{ab}}_*: H^2(\ms Y, \ms H^{\mathrm{ab}})\to H^2(\ms Y, \mathbb{G}_m)$. 
We have $\epsilon^{\mathrm{ab}}_*$ given by the usual cohomology pushforward map for abelian groups, but $\phi_*$ is only defined by the correspondence relation in Giraud. The following lemma explicitly describes this map.
\begin{rmk}
    The locally constant band $\ms H$ is locally a constant sheaf of groups, and the sub-band $\ms H^{\mathrm{der}}$ is given locally by the derived subgroups of these sections of $\ms H$. Since $\ms H$ is a band, these glue together only up to conjugation, so we get a band $\ms H^{\mathrm{der}}$. The band $\ms H^{\mathrm{ab}}$ is constructed similarly, but is locally given by the abelianizations of the sections of $\ms H$. Since all of these groups are abelian, they glue together canonically, forming a sheaf of groups. 
\end{rmk}

\begin{lem}
    Let $S$ be a scheme or algebraic space, and let $\ms Y$ be a locally finitely presented algebraic stack over $S$. Let $\ms H$ be a locally constant finitely presented band on $\ms Y$. Let $\phi: \ms H\to \ms H^{\mathrm{ab}}$ be the abelianization morphism. Then the correspondence 
    $$\phi_*: H^2(\ms Y, \ms H)\to H^2(\ms Y, \ms H^{\mathrm{ab}})$$
    is given by $[\ms X]\mapsto [\ms X \thickslash \ms H^{\mathrm{der}}]$. 
\end{lem}
\begin{proof}
    First, we show that $\ms H$ and $\ms H^{\mathrm{der}}$ are flat and finitely presented. Since these properties may be checked locally and $\ms H$ is locally constant, we may assume that $\ms H$ is a finite constant group scheme, associated to the finite group $H$, so that $\ms H^{\mathrm{der}}$ is the finite constant group scheme associated to the finite group $H^{\mathrm{der}}$, and both $\ms H$ and $\ms H^{\mathrm{der}}$ are flat and finitely presented.   
    
    Let $p:\ms X \to \ms Y$ be a gerbe banded by $\ms H$.  Then Proposition 5.4 implies that $\ms X$ is locally finitely presented, and there exists a flat finitely presented subgroup stack $\ms H'\subseteq \ms I (\ms X)$ such that $\ms X\thickslash \ms H'\cong \ms Y$ and $\underline{\ms H}'\cong \ms H$. Consider $\ms H'^{\mathrm{der}}\subseteq \ms H'\subseteq \ms I(\ms X)$. Clearly $\underline{\ms H'^{\mathrm{der}}}\cong \ms H^{\mathrm{der}}$, so $\ms H '^{\mathrm{der}}$ is also flat finitely presented. 

    Then Corollary 5.6 implies that $\ms H'/\ms H'^{\mathrm{der}}\cong \ms H'^{\mathrm{ab}}\cong \ms H ^{\mathrm{ab}}$ is a flat finitely presented subgroup stack of $\ms X\thickslash \ms H$ and we have the following commutative diagram,
    $$\xymatrix{\ms X \ar[r]^-{\pi_1} \ar[dr]_-{p}& \ms X\thickslash \ms H'^{\mathrm{der}} \ar[d]^-{\pi_2} \\ & \ms Y },$$
    so that indeed $\ms X\thickslash \ms H'^{\mathrm{der}}$ is an $\ms H^{\mathrm{ab}}$-gerbe over $\ms Y$, and $\pi_1:\ms X\to \ms X\thickslash\ms H'^{\mathrm{der}}$ is a $\phi$-morphism.     
\end{proof}

This gives us a sequence
\begin{align}
\xymatrix{ 0\ar[r] & \Pic(\ms{Y}) \ar[r]^-{\pi^*} & \Pic(\ms{X}) \ar[r]^-{\chi} & \Hom(\ms H, \mathbb{G}_m) \ar[r]^-{(-)_*[\ms X]} & H^2(\ms{Y},\mathbb{G}_m)}.
\end{align}

We will show that sequence (6.2.1) is exact, first reducing to the case when $\ms H$ is abelian. Using notation from Lemma 6.2 and its proof, consider the diagram with the top row exact. 

\begin{align}
    \xymatrix{0\ar[r] & \Pic(\ms Y) \ar[r]^-{\pi_2^*} \ar[d]^-{=} & \Pic(\ms X\thickslash \ms H'^{\mathrm{der}}) \ar[r]^-{\chi} \ar[d]^-{\pi_1^*} & \Hom(\ms H^{\mathrm{ab}}, \mathbb{G}_m) \ar[rr]^-{(-)_*[\ms X\thickslash \ms H'^{\mathrm{der}}]}  \ar[d]^-{(-)\circ \phi} & &H^2(\ms Y, \mathbb{G}_m) \ar[d]^-{=} \\ 0\ar[r] & \Pic(\ms Y) \ar[r]^-{p^*} & \Pic(\ms X) \ar[r]^-{\chi} & \Hom( \ms H, \mathbb{G}_m) \ar[rr]^-{(-)_*[\ms X]} & &H^2(\ms Y, \mathbb{G}_m)}
\end{align}

The first square commutes by the diagram above, and the  third squares commutes by the discussion describing $(-)_*[\ms X]$. We show that the second square also commutes. Let $L$ be a line bundle on $\ms X \thickslash \ms H'^{\mathrm{der}}$. To define $\chi_{\pi_1^*L}$, we have an automorphism $\chi_{\pi_1^*L}(a): t^*\pi_1^*L\cong t^*\pi_1^*L$ for every $t:T\to \ms X$ and automorphism $a\in \uAut_T(t)$. To define $\chi_{L}$, we have an automorphism $\chi_{L}(a): s^*L\cong s^*L$ for every $s:T\to \ms X \thickslash \ms H'^{\mathrm{der}}$ and automorphism $a\in\uAut_T(s)$. Here, we view $\uAut_T(t)$ in $\ms I_{\ms X/\ms Y}$ and $\uAut_T(s)$ in $\ms I_{(\ms X\thickslash \ms H^{\mathrm{der}})/\ms Y}$. Since $\pi_1:\ms X\to \ms X\thickslash\ms H'^{\mathrm{der}}$ is a gerbe, there exists a covering $s_i:T_i\to T$ such that $s_i^*s$ lifts to $\ms X$, and call the lift $t_i$. We have $s_i^*s=\pi_1\circ t_i$. Then precomposing the character $\chi_L$ with $\phi$ corresponds to precomposing with $\ms H_{T_i}\cong \underline{\uAut_{T_i}(t_i)}\to \underline{\uAut_T(\pi_1\circ t)}\cong \ms H_{T_i}^{\mathrm{ab}}$, letting automorphisms $a\in\uAut_{T_i}(t_i)$ in $\ms I_{\ms X/\ms Y}$ act on $t_i^*\pi_1^*L$.  Therefore $\chi_L\circ \phi$ is $\chi_{\pi_1^*L}$, and the diagram commutes.   

The next proposition will show that $\pi_1^*$ is an isomorphism.

\begin{prop}
Let $S$ be a locally noetherian scheme, $\ms Y$ a locally finitely presented stack over $S$, and $\ms H$ be a tame locally constant band on $\ms Y$. Let $\pi:\ms{X}\to \ms{Y}$ be a gerbe banded by $\ms H$. The pullback functor
$$\pi^*:(\mathrm{line} \ \mathrm{bundles} \ \mathrm{on} \ \ms{Y})\to(\mathrm{line} \ \mathrm{bundles} \ \mathrm{on} \ \ms{X})$$
induces an equivalence of categories between line bundles on $\ms{Y}$ and line bundles $\ms{L}$ on $\ms{X}$ such that for all geometric points $x:\Spec k \to \ms{X}$, the representation of $\ms H_x=\underline{\uAut_k(x)}$ on $x^*\ms{L}$ is trivial. 
\end{prop}

\begin{proof}
    Using ideas from the proof of Lemma 3.3 \cite[6.2]{Picequiv},
    we will check that for a line bundle $L$ on $\ms{Y}$, $\pi_*\pi^*L\cong L$, and for a line bundle $\ms{L}$ on $\ms{X}$ as above, $\pi^*\pi_*\ms{L}\cong \ms{L}$. We can check this locally, so we can assume that $\ms{X}=B_{\ms{Y}}\ms H$, so that we have a section $\ms Y\to B_{\ms Y}\ms H$, and furthermore that $\ms Y=[\Spec A/G]$ for some local ring $A$ and finite group $G$ (See Theorem 2.9 and Remark 2.10 \cite[11.3.1]{Olssonbook}), and $\ms H =H$ a linearly reductive constant group.

    A line bundle on $\ms Y$ is then a 1-dimensional $A$-module $M$ with semi-linear $G$-action, and a line bundle on $\ms X$ is a 1-dimensional $A$-module $M$ with semi-linear $G\times H$-action, with $H$ acting trivially on $A$. The pullback and pushforward maps are both given by $M\mapsto M$ with the $G\times H$-action induced by its projection to $G$, and $G$-action induced by its inclusion into $G\times H$, respectively. It is then straightforward that $\pi_*\pi^*L\cong L$. 

    Let $\mathfrak{m}_A$ be the maximal ideal of $A$ and write $\ms L$ as a 1-dimensional $R$-module $M=R\cdot e$, with $G\times H$-action $(g,h)\cdot m$ satisfying $(g,h)\cdot \overline{m}=(g,1_H)\cdot \overline{m}$, for $\overline{m}\in M/\mathfrak{m}_A M$. Since $H$ is linearly reductive, the surjection $M\to M/\mathfrak{m}_A M$ remains surjective after taking invariants under $1\times H\cong H$, so we have a surjection 
    $$M^{1\times H}\to M/\mathfrak{m}_A M^{1\times H}\cong M/\mathfrak{m}_A M.$$
    Thus, we can lift the basis element $\overline{e}$ of $M/\mathfrak{m}_A M$ to $M^{1\times H}$ so that we can write $M$ with basis element $e$ invariant under $H$. Then, any $m\in M$ can be written $a\cdot e$ so that the action $(g,h)\cdot (a\cdot e)= (g\cdot a)(g\cdot e)$. It is now also straightforward to check that $\pi^*\pi_*\ms L\cong \ms L$, and the result follows. 
\end{proof}

Now let $\ms{L}\in\Pic(\ms{X})$. For any geometric point $x:\Spec k\to \ms{X}$, we have an action of $\uAut_k(x)\cong \ms{H}_x^{\prime \mathrm{der}}$ on $x^* \ms{L}$ inducing the character $\chi_{x^* \ms{L}}: \ms{H}_x^{\prime \mathrm{der}} \to \mathbb{G}_m$ which must be the zero map since $\mathbb{G}_m$ is abelian. Thus, $\ms{L}$ satisfies the condition that for all geometric points $x:\Spec k \to \ms{X}$, the representation of $\uAut_k(x)$ corresponding to $x^* \ms{L}$ is trivial, so by Proposition 6.3, $\pi_1^*$ is an equivalence $\Pic(\ms{X}\thickslash \ms H ^{\prime \mathrm{der}})\cong \Pic(\ms{X})$. 

Now every vertical arrow in (6.6.2) is an isomorphism, so the bottom sequence is also exact, and we have reduced to the case where $\ms H$ is a sheaf of abelian groups on $\ms Y$. 

After reduction to the abelian case, Proposition 6.3 also implies that sequence (6.2.1) is exact on the left and in the middle. For exactness at $\Pic(\ms{X})$, we need 
$$\pi^*(\Pic(\ms{Y}))=\{\ms{L}\in\Pic(\ms{X}): \chi_{\ms{L}} \ \mathrm{is} \ \mathrm{trivial}\}.$$
For a line bundle $\ms{L}$ on $\ms{X}$ such that $\chi_{\ms{L}}$ is trivial, we have, for each geometric point $x:\Spec k \to \ms{X}$, the representation of $\ms H_x$ is trivial, so by Proposition 6.3, 
$\ms{L}\in\pi^*(\Pic(\ms{Y}))$. Conversely, if $\ms{L}\in \pi^*(\Pic(\ms{Y}))$, then the action of $\ms H_T$ on $t^*\ms{L}$ for any $T$-point $t:T\to \ms{X}$ is trivial. If there were some $t$ such that $\ms H_T$ acted non-trivially on $t^*\ms L$, then there would exist some geometric point $x:\Spec k\to \ms X$ which factors through $t:T\to \ms X$ such that the representation of $\ms H_x$ on $x^*\ms L$ is nontrivial, contradicting Proposition 6.3. 

Finally, for exactness at $\Hom(\ms H,\mathbb{G}_m)$, we need to show that the image of $\chi$ consists of characters $\epsilon:\ms H\to\mathbb{G}_m$ such that $\epsilon_*([\ms{X}])=[B_\ms{Y}(\mathbb{G}_m)]$. The following lemma shows that a line bundle $\ms{L}$ on $\ms{X}$ with character $\epsilon$ corresponds to a morphism of stacks $\ms{X}\to B_{\ms{Y}}\mathbb{G}_m$ with map of stabilizers given by $\epsilon$, so that $\epsilon_*[\ms{X}]=B_{\ms Y}(\mathbb{G}_m)$.  
\begin{lem} 
    Let $\ms{H}$ be an abelian group scheme, $p:\ms{X}\to \ms{Y}$ be an $\ms{H}$-gerbe, and let $\epsilon:\ms{H}\to \mathbb{G}_m$ be a homomorphism of group schemes. Giving a line bundle $\ms{L}$ on $\ms{X}$ with character $\chi_{\ms{L}}=\epsilon$ is equivalent to giving an $\epsilon$-morphism of gerbes (Definition 4.9) $\ms{X}\to B_{\ms{Y}}(\mathbb{G}_m)$.
\end{lem}

\begin{proof}
   Given a line bundle $\ms{L}$ on $\ms{X}$, let $T$ be a scheme, and $t:T\to \ms{X}$ a $T$-point of $\ms{X}$. Then $t^*\ms{L}$ is a line bundle on $T$, so is a $\mathbb{G}_m$-torsor. Since  $B_{\ms{Y}}(\mathbb{G}_m)$ classifies $\mathbb{G}_m$-torsors on $\ms Y$, $t^*\ms L$ is a section of $B_{\ms{Y}}(\mathbb{G}_m)(T)$ . Thus, we get a map $\ms{X}\to B_{\ms{Y}}(\mathbb{G}_m)$. Furthermore, the line bundle having character $\epsilon$ means that for each $T$-point $t:T\to\ms{X}$, $\chi_{t^*\ms{L}}: \ms{H}_T\to \mathbb{G}_m$ which determines the action of $\uAut_T(t)=\ms{H}_T$ on $t^*\ms{L}$ is given by $\epsilon$. In particular, for $h\in \ms{H}_T$, $\chi_{\ms{L}}(h): t^*\ms{L}\to t^*\ms{L}$ is given by multiplication by $\epsilon(h)$. Therefore, our morphism of stacks $\ms{X}\to B_{\ms{Y}}(\mathbb{G}_m)$ has morphism of stabilizer groups given by $\epsilon$. Conversely, a morphism of stacks $\ms{X}\to B_{\ms{Y}}(\mathbb{G}_m)$ corresponds to a line bundle on $\ms{X}$, and the morphism of stacks having its morphism of stabilizer groups $\epsilon$ corresponds to the line bundle having character $\epsilon.$
\end{proof}
This completes the proof of Theorem 1.2.
\qed


\section{Stacky Curves as Gerbes}
In this section, we prove Proposition 1.5 and Corollary 1.7, using rigidification to realize any stacky curve $\ms X$ as a gerbe over a stacky curve with trivial generic stabilizer.

Let $\ms{X}$ be a stacky curve with coarse moduli space $\pi:\ms{X}\to X$ and inertia stack $\ms{I}$. Since $\phi: \ms{I}\to \ms{X}$ is finite, it is affine, so we can write $\ms{I}=\uSpec_{\ms X} (\ms{A})$, where $\ms{A}=\phi_*\ms{O}_{\ms{I}}$ is a coherent sheaf of $\ms{O}_\ms{X}$-algebras on $\ms{X}$, under the equivalence of categories \cite[10.2.4]{Olssonbook}
$${\uSpec}_{\ms{X}}(-): (\mathrm{sheaves} \ \mathrm{of} \ \mathrm{algebras} \ \mathrm{in} \ \Qcoh(\ms{X}))^{op} \to (\mathrm{affine} \ \mathrm{morphisms} \ f:\ms{Y}\to \ms{X}),$$
where $\ms{A}\mapsto {\uSpec}_\ms{X} (\ms{A})\to \ms{X}$, and the inverse functor is given by $(f:\ms{Y}\to\ms{X})\mapsto f_*\ms{O}_{\ms{Y}}.$

Let $j:\ms{U}\to \ms{X}$ be the inclusion of a dense open substack where the inertia group $\ms{I}_\ms{U}$ is finite flat over $\ms{U}$. This exists by \cite[06RB]{SP} We have the following cartesian diagram
\begin{align}
\xymatrix{\ms{I}_\ms{U}={\uSpec}_\ms{U}(\ms{A}_\ms{U})\ar[r]^-{j'} \ar[d]_-{\phi'} & {\uSpec}_\ms{X}(\ms{A})=\ms{I} \ar[d]^-{\phi} \\ \ms{U} \ar[r]^-{j} & \ms{X}}
\end{align}
where $\ms{A}_\ms{U}$ can be described as follows. From the equivalence of categories above and the fact that $j'^*\ms{O}_{\ms{I}}=\ms{O}_{\ms{I}_{\ms{U}}}=j^*\ms A$, we see that $\ms{A}_\ms{U}=\phi'_*j'^*\ms{O}_\ms{I}=j^*\phi_*\ms O_{\ms I}$. 
We have that $\phi'$ is finite flat, so $\ms{A}_\ms{U}$ is a locally free sheaf of $\ms{O}_\ms{U}$-modules of finite rank on $\ms{U}$.

The morphism $j':\ms{I}_{\ms{U}}\to \ms{I}$ is an embedding with image in $\ms{I}$ given by $j'(\ms{I}_\ms{U})={\uSpec}_{\ms{X}}(j_*\ms{A}_{\ms{U}})$. Note that the inclusion of the image $j'(\ms{I}_\ms{U})\to \ms{I}$ is induced by the natural adjunction map $\ms{A}\to j_*j^*\ms{A}$. Let $\ms{H}$ be the closure of $j'(\ms{I}_\ms{U})$ in $\ms{I}$. Then $\ms{H}={\uSpec}_{\ms{X}} (\overline{\ms{A}})$, where $\overline{\ms{A}}=\ms{A}/\mathrm{ker}(\ms{A}\to j_*\ms{A}_\ms{U})$. For $\ms{H}$ to be finite flat over $\ms{X}$, we must have $\overline{\ms{A}}$ be a locally free sheaf of $\ms{O}_\ms{X}$-modules of finite rank on $\ms{X}$. 

We have the following commutative diagram.
$$\xymatrix{\ms{I}_\ms{U} \ar[r] \ar[d] & j'(\ms{I}_\ms{U}) \ar[r] \ar[drr] & \ms{H} \ar[r] \ar[dr] & \ms{I} \ar[d]^-{\phi} \\ \ms{U} \ar[rrr]^-{j} &&& \ms{X}}$$

In particular, we have the following commutative diagram of stacks affine over $\ms{X}$, 
$$\xymatrix{& \ms{H}= {\uSpec}_\ms{X}(\overline{\ms{A}}) \ar[dr] & \\ j'(\ms{I}_\ms{U})={\uSpec}_\ms{X}(j_*\ms{A}_\ms{U})\ar[ur] \ar[rr] && \ms{I}={\uSpec}_\ms{X}(\ms{A}) },$$
induced from the following commutative diagram of quasi-coherent $\ms{O}_{\ms{X}}$-algebras on $\ms{X}$,
$$\xymatrix{ & \overline{\ms{A}}\ar[dl] & \\ j_*\ms{A}_\ms{U}  && \ms{A} \ar[ll] \ar[ul]}.$$

It suffices to check that $\ms{H}$ is flat after base change to $\Spec \ms{O}_{\ms{X},x}$ because the property of being flat commutes with flat base change. After base change with the inclusion $j_x: \Spec \ms{O}_{\ms{X},x} \to \ms{X}$, our diagram is as follows,
$$\xymatrix{\ms{I}_{\ms{U}_{(x)}} \ar[r] \ar[d] & j'(\ms{I}_{\ms{U}_{(x)}}) \ar[r] \ar[drr] & \ms{H}_{(x)} \ar[r] \ar[dr] & \ms{I}_{(x)} \ar[d] \\ \ms{U}_{(x)} \ar[rrr]^-{j} &&& \Spec \ms{O}_{\ms{X},x} }.$$
As with the base change to $\ms{U}$, the inertia stacks in the diagram are all affine over their bases. In particular, we have 
\begin{align*}
\ms{I}_{(x)}=& \uSpec_{\ms{O}_{\ms{X},x}}(\ms{A}_x), \ \mathrm{where} \  \ms{A}_x=j_x^*\ms{A}, \\
\ms{I}_{\ms{U}_{(x)}}=& \uSpec_{\ms{U}_{(x)}}(\ms{A}_{\ms{U}_{(x)}}),\  \mathrm{where} \  \ms{A}_{\ms{U}_{(x)}}= j^*\ms{A}_x, \\
j'(\ms{I}_{\ms{U}_{(x)}})=& \uSpec_{\ms{O}_{\ms{X},x}}(j_*\ms{A}_{\ms{U}_{(x)}}), \\
\ms{H}_{(x)}=& \uSpec_{\ms{O}_{\ms{X},x}}(\overline{\ms{A}_x}),\  \mathrm{where} \  \overline{\ms{A}_x}= \ms{A}_x/\mathrm{ker}(\ms{A}_x\to j_* \ms{A}_{\ms{U}_{(x)}}).
\end{align*}

To ease the notation, write $\ms{O}_{\ms{X},x}=A$, a local ring of dimension 1.
Then since $\ms{X}$ is smooth, $A$ is also regular, so it is a discrete valuation ring. Then $\ms{U}_{(x)}$ is either $\Spec A$ or $\Spec K$, where $K$ is the field of fractions of $A$, so assume $\ms U_{(x)}=\Spec K$. Also, $\ms{A}_x=M$ for some finite $A$-algebra $M$. 
Letting $\pi\in A$ be the uniformizer of $A$, so that $\pi$ generates the maximal ideal $\mathfrak{m}$ of $A$, we have $M=\oplus_{i=1}^n A/\pi^{e_i} \oplus A^r$ for some $e_i,r\in \Z$. Then $\ms{A}_{\ms{U}_{(x)}}= M\otimes_A K= K^r$ as a $K$-algebra, and $j_*\ms{A}_{\ms{U}_{(x)}}=K^r$ as an $A$-algebra, and $\overline{\ms{A}_x}= A^r$. Clearly, $A^r$ is a finite free $A$-module, so $\ms{H}_{(x)}$ is finite flat over $\Spec \ms{O}_{\ms{X},x}$. Thus, $\ms{H}$ is finite flat over $\ms{X}$. 

Since $\ms U$ is dense in $\ms X$, $\ms H$ also has the property that for any open set $\ms V\subseteq \ms X$, the dense open $\ms V'=\ms U\cap \ms V\subseteq \ms V$ has inertia group $\ms I_{\ms V'}=\ms H_{\ms V'}$. Moreover, since $\ms H$ is a subgroup of $\ms I$, it is unramified, so it is \'etale. 

Now, we can apply Theorem 5.1 to get a locally finitely presented algebraic stack $\ms Y:= \ms X\thickslash \ms H$ and a morphism $\pi:\ms X\to \ms X\thickslash \ms H$ which is an fppf-gerbe banded by $\underline{\ms H}$ and since $\ms H$ is finite and \'etale, $\pi$ is proper and \'etale. We show that $\ms Y$ is a stacky curve with trivial generic stabilizer. Since $\ms X$ is smooth, proper, and geometrically connected, so is $\ms Y$. An \'etale cover $U\to \ms X$ gives us an \'etale cover of $\ms X\thickslash \ms H$ after composing with $\pi$ since $\pi$ is \'etale, thus $\ms Y$ is also Deligne-Mumford.

Let $\ms U_{\ms Y}\subseteq \ms Y$ be the substack as in Proposition 2.6, characterized by the property that any $t:T\to \ms Y$ factors through $\ms U_{\ms Y}$ if and only if $\uAut_T(t)=\{\mathrm{id}_T\}$. Suppose that $\ms U_{\ms Y}$ is not dense in $\ms Y$. Then there exists an open substack $\ms V\subseteq \ms Y$ such that every point of $\ms V$ has a nontrivial automorphism group, say $\uAut_T(t)=\ms G_t$, and $\ms I_{\ms V}=\ms G$. Let $\widetilde{\ms V}$ be the pullback of $\ms X$ over $\ms Y$ to $\ms V$. Then the inertia of $\widetilde{\ms V}$ is equal to some extension of the groups $\ms G$ and $\ms H$. This contradicts the property that the inertia of $\ms U\cap \widetilde{\ms V}$ is given by $\ms H$. Thus, $\ms Y$ is a stacky curve with trivial generic stabilizer. Finally, if $\ms X$ is tame, then so is $\ms Y$, since the stabilizer groups of $\ms Y$ are subgroups of stabilizer groups of $\ms X$. This proves Proposition 1.5.
\qed

To prove Corollary 1.7 using Theorem 1.2, we need $\ms H$ tame, locally constant, and finitely presented. Since $\ms X$ is tame, so is $\ms H$. Proposition 1.5 shows that $\ms H$ is finite flat, so we show that $\ms H$ is locally constant. To do this, we describe the local picture surrounding the construction of $\ms H$ more explicitly. As in Remark 2.10, we can express $\ms{X}$ as a stack quotient locally, and we have the following cartesian diagram with $\pi:\ms{X}\to X$ the coarse space morphism, 
$$\xymatrix{[\Spec A/G] \ar[r] \ar[d] & \ms{X} \ar[d]^-{\pi} \\ \Spec A^G \ar[r] & X }$$
where $A$ is a local ring and $G$ is a group scheme. 

\begin{prop} \cite[0374]{SP}
    For a stack quotient $[U/G]$ where $U$ is a scheme and $G$ is a group scheme, we can describe the inertia stack as 
    $$\ms{I}([U/G])=\coprod_{g\in G} [U^g/G]$$
    where $U^g=\{u\in U: g\cdot u=u\}$
\end{prop}

Base changing the inertia stack diagram (7.0.1) with $[\Spec A/G]\to \ms{X}$ yields the diagram,
$$\xymatrix{\ms{I}_K=\coprod_{g\in G} [(\Spec K)^g/G] \ar[r] \ar[d]& \coprod_{g\in G} [(\Spec A)^g/G] =\ms{I}_A \ar[d] \\ [\Spec K/G]\ar[r]^-{j} & [\Spec A/G]},$$
where for all $g\in G$, $(\Spec K)^g=\Spec K$ if and only if $(\Spec A)^g=\Spec A$, otherwise $(\Spec K)^g=\emptyset$. Indeed, $(\Spec A)^g=\Spec A$ implies $(\Spec K)^g=\Spec K$, and if $(\Spec K)^g=\Spec K$, then we have $\Spec K\subseteq (\Spec A)^g \subseteq \Spec A$. But because $(\Spec A)^g$ is the fiber product of the diagram
$$\xymatrix{(\Spec A)^g \ar[r] \ar[d] & \Spec A \ar[d]^-{(\mathrm{id}, g)} \\ \Spec A \ar[r]^-{\Delta} & \Spec A \times \Spec A  }, $$
and $\ms{X}$ is separated, so the diagonal is closed,
we have $(\Spec A)^g$ closed in $\Spec A$, so $(\Spec A)^g=\Spec A$.

Let $H=\{g\in G: (\Spec A)^g=\Spec A\}$, a subgroup scheme of $G$. We see that 
$$H=\{g\in G : g=\mathrm{id}_A \in \Aut(A)\} = \mathrm{ker}(G\to \Aut(A)),$$
so $H$ is normal in $G$. 
Then we can write 
$$\ms{I}_K=\coprod_{h\in H} [\Spec K/G]=[\Spec K/G]\times H,$$ 
and we can write the closure of the image of $\ms{I}_K$ in $\ms{I}_A$ as 
$$\ms{H}_A= \coprod_{h\in H} [\Spec A/G]=[\Spec A/G]\times H.$$
Thus, $\ms{H}$ is locally constant. We can now apply Theorem 1.2 to $p:\ms X\to \ms Y$ to get the desired exact sequence. Moreover, since $\ms X$ is tame, $\ms Y$ is also tame, and in particular, we can apply Theorem 1.1 to realize $\Pic(\ms Y)$ as an extension of two explicit groups.
\qed

\section{Examples}
\begin{ex}
Consider the moduli stack of elliptic curves, $\ms{M}_{1,1}$, over a field $k$ of characteristic not equal to 2 or 3. Objects of $\ms{M}_{1,1}$ are pairs $(T, (f:E\to T, e:T\to E))$, where $T$ is a $k$-scheme, and $(E,e)$ is an elliptic curve over $T$, that is, $f$ is a smooth proper morphism with a section $e$ whose geometric fibers are smooth proper genus 1 curves with a point.

A morphism in $\ms{M}_{1,1}$, $(s, g): (T', (E',e'))\to (T, (E,e))$ is a pair $(s,f)$ where $s:T'\to T$ is a $k$-morphism, and $g:E'\to E$ such that the diagram 
$$\xymatrix{ E' \ar[r]^-{g} \ar[d] & E \ar[d] \\ T' \ar[r]^-{s} & T}$$
is cartesian and $g\circ e'=e\circ s$. 

The coarse moduli space of $\ms{M}_{1,1}$ is $\mathbb{A}^1_j=\Spec k[j]$, where $j$ corresponds to the $j$-invariant of $E$. For $j\in \mathbb{A}^1_j$, let $E_j$ be the elliptic curve with $j$-invariant $j$, and for $j\neq 0, 1728$, $\Aut(E_j)=\mu_2$, but for $j=0$, $E_0: y^2=x^3+1$ has $\Aut(E_0)=\mu_6$, and for $j=1728$, $E_{1728}:y^2=x^3+x$ has $\Aut(E_{1728})=\mu_4$. 

We compute the inertia stack $\ms{I}$ of $\ms{M}_{1,1}$. From the discussion above, we see that the inertia stack consists of two copies of $\ms{M}_{1,1}$, where one copy corresponds to the objects of $\ms{M}_{1,1}$ together with the trivial automorphism and the other copy corresponds to the objects of $\ms{M}_{1,1}$ together with the involution. Additionally, we have 4 more points isomorphic to $B\mu_6$, corresponding to the point $E_0$ of $\ms{M}_{1,1}$ together with the remaining automorphisms of $E_0$, $\mu_6 \setminus \mu_2$, and 2 more points isomorphic to $B\mu_4$, corresponding to the point $E_{1728}$ with its remaining automorphisms, $\mu_4 \setminus \mu_2$. Thus, 
\begin{align}
\ms{I}\cong \ms{M}_{1,1}\coprod \ms{M}_{1,1}\coprod \left(\coprod_4 B\mu_6 \right) \coprod \left(\coprod_2 B\mu_4\right). 
\end{align}

Now, we would like to find a finite flat subgroup scheme $\ms{H}$, normal in $\ms{I}$, and take the rigidification of $\ms{M}_{1,1}$ with respect to $\ms{H}$. Note that $\ms{I}$ is already finite over $\ms{M}_{1,1}$, but it is not flat, as the length of the fibers is not constant. We find an open substack $\ms{U}$ of $\ms{M}_{1,1}$ where $\ms{I}(\ms{U})$ is finite flat over $\ms{U}$. From the computation of the inertia stack, we see that we can take $\ms{U}=\ms{M}_{1,1}\setminus \{E_0, E_{1728} \}$. Equivalently, this is the cartesian product of $\ms{M}_{1,1}$ with $U=\mathbb{A}^1_j \setminus \{0, 1728\}$ over $\mathbb{A}^1_j$. Since we have a section $U\to \ms{M}_{1,1}$, $\ms{U}\cong B_U(\mu_2)$. Then $\ms{I}_{\ms{U}}= \ms U\times \mu_2=\ms U\coprod \ms U$ inside the first two components of $\ms I$ as in (8.1.1). Clearly, this is finite flat over $\ms{U}$. Then taking $\ms{H}$ to be the closure of the image of $\ms{I}_{\ms{U}}$ in $\ms{I}$, we get $\ms{H}=\ms{M}_{1,1}\times \mu_2=\ms M_{1,1}\coprod \ms M_{1,1}$, the first two components of $\ms I$ as in $(8.1.1)$, which is also clearly finite flat over $\ms{M}_{1,1}$. Additionally, all the groups appearing in this example have been abelian, so $\ms{H}$ is also normal in $\ms{I}$, and we can now take the rigidification of $\ms{M}_{1,1}$ with respect to $\ms{H}$ to get $\ms{N}_{1,1}$, the root stack of $\mathbb{A}^1_j$ of order 2 at $j=1728$ and order 3 at $j=0$, such that $\pi:\ms M_{1,1}\to \ms N_{1,1}$ is a $\mu_2$-gerbe. 

We get an exact sequence 
$$\xymatrix{0\ar[r] & \Pic(\ms N_{1,1}) \ar[r]^-{\pi^*} & \Pic(\ms M_{1,1}) \ar[r]^-{\chi} & \Hom(\mu_2, \mathbb{G}_m)\ar[rr]^-{(-)_*[\ms M_{1,1}]} && H^2(\ms N_{1,1},\mathbb{G}_m)}.$$

Since we have that $\ms U \cong B_U(\mu_2)$ is a trivial gerbe and $\ms U$ is dense in $\ms M_{1,1}$, we have that $\ms M_{1,1}\cong B_{\ms N_{1,1}}(\mu_2)$ is a trivial gerbe. Thus, the pushforward map $(-)_*[\ms M_{1,1}]$ is the zero map, so in fact the sequence
$$\xymatrix{0\ar[r] & \Pic(\ms N_{1,1}) \ar[r] & \Pic(\ms M_{1,1}) \ar[r] & \Hom(\mu_2, \mathbb{G}_m)=\Z/2\Z \ar[r] & 0}$$
is exact. 

By Theorem 1.1, $\Pic(\ms N_{1,1})$ can be described as the pushout of the following diagram,
$$\xymatrix{\Z^2 \ar[r]^-{(\cdot 2, \cdot 3)} \ar[d]_-{\phi} & \Z^2 \ar[d] \\ \Pic(\mathbb{A}^1_j) \ar[r]^-{\pi^*} & \Pic(\ms N_{1,1})},$$
so $\Pic(\ms N_{1,1})=\Z/2\Z\oplus \Z/3\Z$, since $\Pic(\mathbb{A}^1_j)=0$. Moreover, the ideal sheaves of $\ms N_{1,1}$ at the points $j=1728$ and $j=0$ give nontrivial line bundles whose 4th and and 6th powers, respectively give a line bundle which must come from the coarse space, $\mathbb{A}^1_j$, by Lemma 3.3, with trivial Picard group. Thus $\Pic(\ms M_{1,1})$ has an element of order 4 and order 6, so must be $\Z/12\Z$. 
\end{ex}

\begin{ex}
    Let $P$ be the vanishing locus of the equation $x^2+y^2+z^2$ inside $\mathbb{P}^2_{\mathbb{R}}$, that is, 
    $$P=\Proj(\mathbb{R}[x,y,z]/(x^2+y^2+z^2))$$
    or
    $$P=\left[\dfrac{\Spec(\mathbb{R}[x,y,z])\backslash \{0\}}{(x^2+y^2+z^2)}/\mathbb{G}_m\right].$$
    This curve has no $\mathbb{R}$-points, but the base change over the inclusion $\mathbb{R}\to \mathbb{C}$ 
    $$P_{\mathbb{C}}=\Proj(\mathbb{C}[x,y,z]/(x^2+y^2+z^2))$$
    is isomorphic to $\mathbb{P}^1_{\mathbb{C}}$, with inclusion into $\mathbb{P}_{\mathbb{C}}^2$ corresponding to a line bundle of degree 2, and the diagram
    $$\xymatrix{P_{\mathbb{C}} \ar[r] \ar[d] & P \ar[d] \\ \mathbb{P}_{\mathbb{C}}^2 \ar[r] & \mathbb{P}_{\mathbb{R}}^2 }$$
    commutes. 
    Thus, there exists a line bundle of degree 2 on $P$ defining the inclusion of $P$ into $\mathbb{P}_{\mathbb{R}}^2$, call it $\ms L_P$. However, there is no line bundle of degree 1 on $P$. If there were, it would pull back to the degree 1 line bundle on $\mathbb{P}^1_{\mathbb{C}}$, so it would have two global sections, giving an isomorphism of $P$ with $\mathbb{P}_{\mathbb{R}}^1$, a contradiction.  

    Let $\ms P$ be the stack with objects 
    $$(f:T\to P, (\ms M,\epsilon)),$$ 
    where $T$ is a $P$-scheme, $\ms M$ is a line bundle on $T$, and $\epsilon:\ms M^{\otimes 2}\to f^*\ms L_P$ is isomorphism of line bundles on $T$. Morphisms are pairs 
    $$(t,\sigma):(f':T'\to P, (\ms M',\epsilon'))\to (f:T\to P, (\ms M,\epsilon)),$$
    where $t:T'\to T$ is a morphism of $P$-schemes, and $\sigma:\ms M'\to t^*\ms M$ is an isomorphism of line bundles on $T'$ such that the diagram
    $$\xymatrix{\ms M'^{\otimes 2}\ar[rr]^-{\sigma^{\otimes 2}} \ar[dr]_-{\epsilon'} &&  t^*\ms M^{\otimes 2} \ar[dl]^-{t^*\epsilon} \\ & f'^*\ms L_P & }$$
    commutes. 

    We show that $\ms P$ is $\mu_2$-gerbe over $P$. For any $f:T\to P$, consider the covering $t:T_\mathbb{C} \to T$ coming from the fiber product diagram
    \begin{align}
    \xymatrix{T_{\mathbb{C}} \ar[r]^-{t} \ar[d]^-{g} & T \ar[d]^-{f} \\ P_{\mathbb{C}} \ar[r] & P}.
    \end{align}
    Then $(f\circ t)^*\ms L_P=g^*\ms O_{\mathbb{P}^1_{\mathbb{C}}}(2)=g^*\ms O_{\mathbb{P}_{\mathbb{C}}^1}^{\otimes 2}$, so $(f\circ t_i:T_{\mathbb{C}}\to P, (g^*\ms O_{\mathbb{P}_{\mathbb{C}}^1}, \mathrm{id}))$ is an element of $\ms P(T_{\mathbb{C}})$, and thus $\ms P(T_{\mathbb{C}})$ is nonempty. 

    Let $(\ms M',\epsilon'), (\ms M, \epsilon)\in \ms P(f:T\to P)$. We choose a trivializing cover so that we may assume $\ms M=\ms M'=\ms O_T$. The isomorphism 
    $$\xymatrix{\ms O_T\cong \ms O_T^{\otimes 2}\ar[r]^-{\epsilon} & f^*\ms L_P & \ms O_T^{\otimes 2}\cong \ms O_T\ar[l]_-{\epsilon'}}$$
    is given by multiplication by some $u\in \ms O_T^*.$ \'Etale locally, $u$ has a square root, and on this cover, $(\ms M',\epsilon')\cong (\ms M,\epsilon)$. 

    Finally, for any $(\ms M,\epsilon)\in \ms P(f:T\to P)$, $\uAut_T(\ms M,\epsilon)$ is the set of automorphisms $\sigma:\ms M\to \ms M$ such that $\epsilon=\epsilon\circ \sigma^{\otimes 2}:\ms M^{\otimes 2}\to f^*\ms L_P$, so $\uAut_T(\ms M,\epsilon)=\mu_2$.

    Thus, $\pi:\ms P\to P$ is a $\mu_2$-gerbe, and let $[\ms P]$ be the class of the gerbe $\ms P$ in $H^2(P,\mu_2)$. Since $\ms L_P$ does not have a global square root on $P$, this gerbe is nontrivial. 
    
    By Theorem 1.2, we have an exact sequence
    $$\xymatrix{0\ar[r] & \Pic(P) \ar[r]^-{\pi^*} & \Pic(\ms P) \ar[r]^-{\chi} & \Hom(\mu_2,\mathbb{G}_m) \ar[r]^-{(-)_*[\ms P]} & H^2(P,\mathbb{G}_m)}.$$
    Although $\ms P$ is not the trivial gerbe, the map $(-)_*[\ms P]$ is still the zero map. Clearly the trivial morphism $1:\mu_2\to \mathbb{G}_m$ is such that $1_*[\ms P]=B_P(\mathbb{G}_m)$. For the inclusion $\iota:\mu_2\to \mathbb{G}_m$, we have the short exact sequence 
    \begin{align}
    \xymatrix{0\ar[r] & \mu_2 \ar[r]^-{\iota} & \mathbb{G}_m \ar[r]^-{\alpha} & \mathbb{G}_m \ar[r] & 0},
    \end{align}
    with $\alpha:x\mapsto x^2$, which gives us the exact sequence
    $$\xymatrix{\ & \ar[r] H^1(P,\mathbb{G}_m) \ar[r]^-{\alpha} & H^1(P,\mathbb{G}_m) \ar[r]^-{\delta} & H^2(P,\mu_2) \ar[r]^-{\iota_*} & H^2(P,\mathbb{G}_m)}.$$
    \begin{lem}
    The image of $\ms L_P\in H^1(P,\mathbb{G}_m)$ under $\delta$ is $[\ms P]$.
    \end{lem}
     \begin{proof}
        The $\mathbb{G}_m$-torsor over $P$ associated to $\ms L_P$ is the principle homogeneous space $\mathbb{L}_P$, where $$\mathbb{L}_P=\uSpec_P(\oplus_{i\in\Z}\ms L_P^{\otimes i}),$$ with action of $\mathbb{G}_m$ given as follows. 

        Write 
        \begin{align*}
        \mathbb{L}_P(f:T\to P)=& \{\rho:f^*(\oplus_{i\in\Z} f^*\ms L_P^{\otimes i}) \to \ms O_T\} \\
        =& \{\oplus_{i\in\Z} \rho_i: \otimes_{i\in\Z} f^*\ms L_P ^{\otimes i} \to \ms O_T\}.
        \end{align*}
        Then $u\in\mathbb{G}_m(T)$ acts on $\rho\in \mathbb{L}_P(T)$ by $u\cdot \rho=\oplus_{i\in\Z} u^i \rho_i$.

        This indeed defines a $\mathbb{G}_m$-torsor. For any $f:T\to P$, take the covering $t:T_{\mathbb{C}}\to T$ from the fiber diagram $(8.2.1)$, $(f\circ t)^* \ms L_P^{\otimes i}\cong g^*\ms P_{\mathbb{P}_{\mathbb{C}}^1}^{\otimes 2i}$. Then the multiplication map $\oplus_{i\in\Z}(f\circ t)^*\ms L_P^\otimes i\to \ms O_{T_{\mathbb{C}}}$ is an element of $\mathbb{L}_P(f\circ t:T_{\mathbb{C}}\to P)$, which is thus nonempty.

        Given two elements $\rho,\rho': \oplus_{i\in\Z} f^*\ms L_P^{\otimes i}\to \ms O_T$ in $\mathbb{L}_P(f:T\to P)$, there exists a unique $u\in\mathbb{G}_m(T)$ such that $\rho_1=u\rho_1':f^*\ms L_P\to \ms O_T$, and this is compatible with the tensor product structure, so also $\rho=u\cdot \rho': \oplus_{i\in\Z}f^*\ms L_P^{\otimes i}\to \ms O_T$.

        Now let $\ms G$ be the image of $\mathbb{L}_P$ under the boundary map $\delta$ \cite[12.2.5]{Olssonbook}. This is the $\mu_2$-gerbe with objects 
        $$(f:T\to P,\ \mathbb{L}, \ \sigma: \alpha_*\mathbb{L}\to \mathbb{L}_P|_T),$$
        where $T$ is a $P$-scheme, $\mathbb{L}$ is a $\mathbb{G}_m$-torsor on $T$, and $\sigma$ is an isomorphism of $\mathbb{G}_m$-torsors on $T$.
        Morphisms are pairs 
        $$(t,\tau): (f':T'\to P,\ \mathbb{L}',\ \sigma': \alpha_*\mathbb{L}'\to \mathbb{L}_P|_T)\to (f:T\to P, \ \mathbb{L}, \ \sigma: \alpha_*\mathbb{L}\to \mathbb{L}_P|_T)$$
        where $t:T'\to T$ is a $P$-morphism, and $tau$ is an isomorphism of $\mathbb{G}_m$-torsors such that the diagram 
        $$\xymatrix{\alpha_*\mathbb{L}'\ar[rr]^-{\alpha_*\tau} \ar[dr]^-{\sigma'} && t^*\alpha_*\mathbb{L}\ar[dl]_-{t^*\sigma} \\ & \mathbb{L}_P |_T &}$$
        commutes.
        We define a morphism of $\mu_2$-gerbes $\phi:\ms P\to \ms G$ to conclude that $\ms P\cong \ms G$ \cite[12.2.4]{Olssonbook}. 

        Let $\phi$ be the map given by 
        $$(f:T\to P, (\ms M,\epsilon))\mapsto(f:T\to P,\ \mathbb{M}, \ \sigma_{\epsilon}:\alpha_*\mathbb{M}\cong \mathbb{L}_P|_T),$$
        where $\mathbb{M}=\uSpec_T(\oplus_{i\in\Z} \ms M^{\otimes i})$, and $\mathbb{G}_m$ acts analogously as on $\mathbb{L}_P$. 
        To understand $\sigma_{\epsilon}$, we must first understand $\alpha_*\mathbb{M}$. Since $\alpha$ is surjective, $\alpha_*\mathbb{M}$ is given by $\mathbb{M}/ \mu_2$, with action of $\mathbb{G}_m$ given by $u\cdot \overline{\lambda}=\overline{u^{1/2}\lambda}$. Note that the choice of square-root is unambiguous here. Taking the quotient of $\mathbb{M}$ by the $\mu_2$ action makes any element $\lambda: \oplus_{i\in\Z}\ms M^{\otimes i}\to \ms O_T$ well-defined only up to $i\in 2\Z$, so we can write elements of $\mathbb{M}$ as $\lambda':t^*\oplus_{i\in\Z}\ms M^{\otimes 2i}\to \ms O_S$, with $\mathbb{G}_m$-action given by $u\cdot \lambda=\oplus_{i\in\Z}u^i\lambda_i':\oplus_{i\in\Z} t^*\ms M^{\otimes 2i}\to \ms O_S$. Then $\sigma_{\epsilon}$ is simply induced by the isomorphism of line bundles $\epsilon:\ms M^{\otimes 2}\to f^*\ms L_P$, and $\sigma_{\epsilon}$ is an isomorphism of $\mathbb{G}_m$-torsors as it is clearly $\mathbb{G}_m$-equivariant. 

        A morphism $(t,\sigma)$ in $\ms P$ maps to the morphism $(t,\tau)$ with $\tau$ induced by $\sigma$. This is clearly a morphism of $\mu_2$-gerbes, so is an isomorphism. 
        \end{proof}
    By the lemma, it follows that $\iota_*[\ms P]=B_P(\mathbb{G}_m)$, and we conclude that 
    $$(-)_*[\ms P]: \Hom(\mu_2,\mathbb{G}_m)\to H^2(P,\mathbb{G}_m)$$
    is the zero map and that
     and the sequence
     $$\xymatrix{0 \ar[r] & \Pic(P)\cong \Z \ar[r]^-{\pi^*} & \Pic(\ms P) \ar[r]^-{\chi} & \Hom(\mu_2,\mathbb{G}_m)\cong \Z/2\Z \ar[r] &0}$$
     is exact.    

    Using the Yoneda lemma, the identity map $\mathrm{id}:\ms P\to \ms P$ corresponds to the object $(\pi:\ms P\to P, (\ms L_{\ms P}, \epsilon_{\ms P}))$, with $\ms L_{\ms P}$ a universal square root of $\ms L_P$ on $\ms P$. Thus, we have a line bundle on $\ms P$ whose square is the generator of $\Pic(P)$. Thus $\ms L_{\ms P}$ generates $\Pic(\ms P)$, and $\Pic(\ms P)\cong \Z$.   
\end{ex}
\begin{ex}    
    Using the same set-up as in the previous example, now let $\ms P'=B_P(\mu_2)$ be the trivial $\mu_2$-gerbe over $P$. We have the same short exact sequence
     $$\xymatrix{0 \ar[r] & \Pic(P)\cong \Z \ar[r]^-{\pi^*} & \Pic(\ms P') \ar[r]^-{\chi} & \Hom(\mu_2,\mathbb{G}_m)\cong \Z/2\Z \ar[r] &0},$$
     but $\Pic(\ms P')\cong \Z\times \Z/2\Z$. 
\end{ex}
\begin{rmk}
    The stacky curves in the last two examples have different Picard groups, but the Picard groups are given by an extension of the same two groups. This exemplifies the fact that the Picard group of a stacky curve depends on its structure as a gerbe. 
\end{rmk}
\begin{ex}
    Let $(E,e)$ be an elliptic curve over a field $k$ with structure morphism $f:E\to \Spec k$ and section $e:\Spec k\to E$. The section $e$ induces a pullback map $e^*: H^2(E,\mathbb{G}_m) \to H^2(k,\mathbb{G}_m)$ which splits the pullback of the structure morphism $f^*: H^2(k,\mathbb{G}_m)\to H^2(E,\mathbb{G}_m)$, and similarly for $H^3$. 

    We can use the Leray Spectral Sequence 
    $$E_2^{p,q}=H^p(k,R^qf_*\mathbb{G}_m)\Rightarrow H^{p+q}(E,\mathbb{G}_m)$$
    to compute $H^2(E,\mathbb{G}_m)$. We have $R^0f_*\mathbb{G}_m=\mathbb{G}_m$, $R^1f_*\mathbb{G}_m=\Pic_{E/k}$, the relative Picard scheme of $E$ over $k$, and $R^if_*\mathbb{G}_m=0$ for all $i\geq 2$ by Tsen's Theorem. Then the spectral sequence results in an exact sequence describing $H^2(E,\mathbb{G}_m)$,
    $$\xymatrix{\Pic_{E/k}(k)\ar[r]^-{\phi} & H^2(k,\mathbb{G}_m) \ar[r]^{f^*} & H^2(E,\mathbb{G}_m) \ar[r] & H^1(k, \Pic_{E/k}) \ar[r]^{\phi'} & H^3(k, \mathbb{G}_m)},$$
    and a filtration
    $$\xymatrix{0\ar[r] & H^2(k,\mathbb{G}_m)/\mathrm{im}(\phi) \ar[r]& H^2(E,\mathbb{G}_m) \ar[r] & \ker(\phi') \ar[r] & 0}.$$
    Similarly, for $H^3(E,\mathbb{G}_m)$, we have an exact sequence 
    $$\xymatrix{H^1(k,\Pic_{E/k})\ar[r]^-{\phi'} & H^3(k,\mathbb{G}_m) \ar[r]^-{f^*} & H^3(E,\mathbb{G}_m}).$$
    Putting the two sequences together, we have an exact sequence
     $$\xymatrix{\Pic_{E/k}(k)\ar[r]^-{\phi} & H^2(k,\mathbb{G}_m) \ar[r]^{f^*} & H^2(E,\mathbb{G}_m) &&}$$ $$\xymatrix{&&\ar[r] & H^1(k, \Pic_{E/k}) \ar[r]^{\phi'} & H^3(k, \mathbb{G}_m) \ar[r]^-{f^*} & H^3(E,\mathbb{G}_m)}.$$
     The section $e:\Spec k\to E$ gives us two splittings of this sequence so that both $f^*$ maps are injective. We conclude that $\phi=\phi'=0$, and 
     $$\xymatrix{0 \ar[r] & H^2(k,\mathbb{G}_m) \ar[r]^-{f^*} & H^2(E,\mathbb{G}_m) \ar[r] & H^1(k,\Pic_{E/k}) \ar[r] & 0}$$ 
     is a split exact sequence. 

     Now again using the Kummer sequence (8.2.2), we construct an example of a stacky curve $\ms X$ which is a $\mu_2$-gerbe over $E$ with $\chi:\Pic(\ms X)\to \Hom(\mu_2,\mathbb{G}_m)$ not surjective, unlike the previous examples. Let $k=\mathbb{R}$, so that $H^2(k,\mathbb{G}_m)=\mu_2$, with nontrivial element the class of the Hamiltonian algebra $\mathbb{H}$. Let $[\ms G_{\mathbb{H}}]$ be the image of $[\mathbb{H}]$ under the inclusion $f^*:H^2(k,\mathbb{G}_m)\to H^2(E,\mathbb{G}_m)$. 
     Since $[\mathbb{H}]$ has order 2 in $H^2(k,\mathbb{G}_m)$, $\alpha_*[\ms G_{\mathbb{H}}]=[B_E(\mathbb{G}_m)]$ in $H^2(E,\mathbb{G}_m)$. Thus, $[\ms G_{\mathbb{H}}]$ came from an element $[\ms X]\in H^2(E,\mu_2)$. Then $\ms X$ is a stacky curve which is a $\mu_2$-gerbe over $E$, and we have an exact sequence 
     $$\xymatrix{0 \ar[r] & \Pic(E) \ar[r] & \Pic(\ms X) \ar[r]^-{\chi} & \Hom(\mu_2,\mathbb{G}_m) \ar[r]^-{(-)_*[\ms X]} & H^2(E,\mathbb{G}_m)}.$$
     For the inclusion $\iota:\mu_2\to \mathbb{G}_m$, we have $\iota_*[\ms X]=[\ms G_{\mathbb{H}}]$, so $(-)_*$ is injective, thus $\chi$ is the zero map, and $\Pic(\ms X)\cong \Pic(E)$. In this example, the gerbe structure of $\ms X$ did not contribute the the Picard group of $\ms X$ over $E$. We say that there does not exist any twisted sheaves on $\ms X$. See \cite{twisted} for more on twisted sheaves.
\end{ex}

\bibliographystyle{alpha}
\bibliography{bibliography}

\begin{thebibliography}{{Sta}23}

\bibitem[AOV08]{Tame}
Dan Abramovich, Martin Olsson, and Angelo Vistoli.
\newblock Tame stacks in positive characteristic.
\newblock {\em Ann. Inst. Fourier (Grenoble)}, 58(4):1057--1091, 2008.

\bibitem[FO10]{PicM}
William Fulton and Martin Olsson.
\newblock The {P}icard group of {$\mathscr M_{1,1}$}.
\newblock {\em Algebra Number Theory}, 4(1):87--104, 2010.

\bibitem[Gir71]{Giraudbook}
Jean Giraud.
\newblock {\em Cohomologie non ab\'{e}lienne}.
\newblock Die Grundlehren der mathematischen Wissenschaften, Band 179.
  Springer-Verlag, Berlin-New York, 1971.

\bibitem[GS17]{GS}
Anton Geraschenko and Matthew Satriano.
\newblock A ``bottom up'' characterization of smooth {D}eligne-{M}umford
  stacks.
\newblock {\em Int. Math. Res. Not. IMRN}, (21):6469--6483, 2017.

\bibitem[Har77]{Hartshorne}
Robin Hartshorne.
\newblock {\em Algebraic geometry}.
\newblock Graduate Texts in Mathematics, No. 52. Springer-Verlag, New
  York-Heidelberg, 1977.

\bibitem[KM97]{KM}
Se\'{a}n Keel and Shigefumi Mori.
\newblock Quotients by groupoids.
\newblock {\em Ann. of Math. (2)}, 145(1):193--213, 1997.

\bibitem[Knu71]{Knutsonbook}
Donald Knutson.
\newblock {\em Algebraic spaces}.
\newblock Lecture Notes in Mathematics, Vol. 203. Springer-Verlag, Berlin-New
  York, 1971.

\bibitem[Lie07]{twisted}
Max Lieblich.
\newblock Moduli of twisted sheaves.
\newblock {\em Duke Math. J.}, 138(1):23--118, 2007.

\bibitem[Mum65]{Mumfordart}
David Mumford.
\newblock Picard groups of moduli problems.
\newblock In {\em Arithmetical {A}lgebraic {G}eometry ({P}roc. {C}onf. {P}urdue
  {U}niv., 1963)}, pages 33--81. Harper \& Row, New York, 1965.

\bibitem[Nil16]{Niles}
Andrew Niles.
\newblock The {P}icard groups of the stacks {$\mathcal{Y}_0(2)$} and
  {$\mathcal{Y}_0(3)$}.
\newblock {\em Funct. Approx. Comment. Math.}, 55(1):105--112, 2016.

\bibitem[Ols07]{Log}
Martin~C. Olsson.
\newblock ({L}og) twisted curves.
\newblock {\em Compos. Math.}, 143(2):476--494, 2007.

\bibitem[Ols12]{Picequiv}
Martin Olsson.
\newblock Integral models for moduli spaces of {$G$}-torsors.
\newblock {\em Ann. Inst. Fourier (Grenoble)}, 62(4):1483--1549, 2012.

\bibitem[Ols16]{Olssonbook}
Martin Olsson.
\newblock {\em Algebraic spaces and stacks}, volume~62 of {\em American
  Mathematical Society Colloquium Publications}.
\newblock American Mathematical Society, Providence, RI, 2016.

\bibitem[SGA03]{SGA1}
{\em Rev\^{e}tements \'{e}tales et groupe fondamental ({SGA} 1)}, volume~3 of
  {\em Documents Math\'{e}matiques (Paris) [Mathematical Documents (Paris)]}.
\newblock Soci\'{e}t\'{e} Math\'{e}matique de France, Paris, 2003.
\newblock S\'{e}minaire de g\'{e}om\'{e}trie alg\'{e}brique du Bois Marie
  1960--61. [Algebraic Geometry Seminar of Bois Marie 1960-61], Directed by A.
  Grothendieck, With two papers by M. Raynaud, Updated and annotated reprint of
  the 1971 original [Lecture Notes in Math., 224, Springer, Berlin; MR0354651
  (50 \#7129)].

\bibitem[{Sta}23]{SP}
The {Stacks project authors}.
\newblock The stacks project.
\newblock \url{https://stacks.math.columbia.edu}, 2023.

\end{thebibliography}

\end{document}